\documentclass[11pt]{elsarticle}

\usepackage{pifont} % dingbats
\usepackage[stretch=10]{microtype}
\usepackage{xy}
\usepackage{graphicx,amsfonts,latexsym,amssymb,cancel,pstricks,array}
\usepackage{amssymb,amscd,amsmath,latexsym}
\usepackage{hyperref}
\usepackage{slashbox}
\usepackage{tikz}
\tikzset{node distance=2cm, auto}
\usetikzlibrary{matrix}

\input xy
\xyoption{all}
\newcommand{\su}{{\underline s}}

\newcommand{\kapp}{{\bo K}}
\newcommand{\ao}{\alpha_1}
\newcommand{\az}{\alpha_0}

\newcommand{\BBM}{\B^{\mu}}
\newcommand{\BBS}{\B^{\sigma}}
\newcommand{\BBB}{\B^{\mu \sigma}}
\newcommand{\BOM}{\B_1^{\mu}}
\newcommand{\BTM}{\B_2^{\mu}}
\newcommand{\BOMS}{\B_1^{\mu \sigma}}
\newcommand{\BTMS}{\B_2^{\mu \sigma}}
\newcommand{\BO}{\bo {B_1}}
\newcommand{\BT}{\bo {B_2}}
\newcommand{\kap}{\kappa}
\newcommand{\bo}[1]{\mathbf{#1}}
\newcommand{\B}{{\bo{B}}}
\newcommand{\bardan}{\bar{\D}_{\underline s}}

\newcommand{\Hom}{\operatorname{Hom}}
\newcommand{\Homcat}{\operatorname{h}}

\newcommand{\CATC}{\C}

\newcommand{\ra}{\rightarrow}

\newcommand{\FC}{{\bo F_\ast \C}}
\newcommand{\FB}{\overline{\bo F_{\ast} \BBB}}
\newcommand{\FBO}{\overline{\bo F_{\ast} \BOMS}}
\newcommand{\FBT}{\overline{\bo F_{\ast} \BTMS}}
\newcommand{\bFC}{\overline{\FC}}
\newcommand{\lra}{\longrightarrow}
\newcommand{\spaces}{\bo{Spaces}}

\newcommand{\algc}{\bo{Alg^C}}
\newcommand{\algcc}{\bo{Alg^{\CC}}}
\newcommand{\algt}{\bo{Alg^T}}
\newcommand{\T}{{\bo T}}
\newcommand{\C}{{\bo C}}

\newcommand{\D}{{\bo D}}

\newcommand{\Prj}{{\bo P}}
\newcommand{\M}{{\bo M}}

\newcommand{\CC}{{\bar{\C}}}

\newcommand{\sC}{{\spaces^{\C}}}

\newcommand{\sFC}{\spaces^{\FC}}

\newcommand{\sprf}{\spaces^{\Prj}_{fib}}

\newcommand{\lsc}{{\bo L \sC}}
\newcommand{\lsd}{{\bo L \spaces^{\B}}}
\newcommand{\lsddo}{{\bo L \spaces^{\BO}}}
\newcommand{\lsddt}{{\bo L \spaces^{\BT}}}
\newcommand{\lsddm}{{\bo L \spaces^{\BBM}}}
\newcommand{\lsco}{{\bo L \spaces^{\C_1}}}
\newcommand{\sFCO}{\spaces^{\bo F_{\ast}\C_1}}

\newcommand{\map}{{\rm Map}}

\newcommand{\KJ}{(K_{\C},J_{\C})}

\newcommand{\CCC}{\bo{C}}

\newcommand{\Set}{\bo{Sets}}

\newcommand{\id}{{\ensuremath{\rm id}}}
\newcommand{\colim}[1][]{\operatorname{colim}_{#1}}
\newcommand{\hocolim}[1][]{\operatorname{hocolim}_{#1}}
\newcommand{\holim}[1][]{\operatorname{holim}_{#1}}

\def\sp#1{s_{#1}^{1}}

\def\s#1{s_{#1}}

\def\alg#1{{\bo{Alg}}^{#1}}

\def\mapright#1#2{\smash{\mathop{\hbox to
#1pt{\rightarrowfill}}\limits^{#2}}}

\newproof{proof}{Proof}
\newproof{poto}{Proof of Theorem {\ref{mainIV}}}
\newproof{pott}{Proof of Theorem {\ref{MAINI}}}
\newproof{pottc}{Proof of Theorem{\ref{MAINI COR}}}
\newproof{potth}{Proof of Theorem {\ref{MAINII}}}
\newproof{pol}{Proof of Lemma {\ref{compare}}}
\newproof{polo}{Proof of Lemma {\ref{FC UNIT}}}
\newproof{polt}{Proof of Lemma {\ref{BB}}}
\newtheorem{theorem}{Theorem}[section]
\newtheorem{corollary}[theorem]{Corollary}
\newtheorem{lemma}[theorem]{Lemma}
\newtheorem{proposition}[theorem]{Proposition}
\newdefinition{definition}[theorem]{Definition}
\newdefinition{example}[theorem]{Example}

\newtheorem{note}[theorem]{Note}
\newtheorem{remark}[theorem]{Remark}

\begin{document}
%%%%%%%%%%%%%%%%%%%%%%%%%%%%
%   Opening
%%%%%%%%%%%%%%%%%%%%%%%%%%%%
\begin{frontmatter}
\title{Rigidification of Homotopy Algebras over Finite Product Sketches}
\author{Bruce R Corrigan-Salter}
\ead{brcs@wayne.edu}
\address{Department of Mathematics\\ Wayne State University\\
1150 Faculty/Administration Building\\
656 W. Kirby 
\\Detroit, MI~48202, USA}

%%%%%%%%%%%%%%%%%%%%%%%%%%%%
%   End Opening
%%%%%%%%%%%%%%%%%%%%%%%%%%%%

\begin{abstract}
Multi-sorted algebraic theories provide a formalism for describing various structures on spaces that are of interest in homotopy theory. The results of Badzioch (2002) and Bergner (2006) showed that an interesting feature of this formalism is the following rigidification property. Every multi-sorted algebraic theory defines a category of homotopy algebras, i.e. a category of spaces equipped with certain structure that is to some extent homotopy invariant. Each such homotopy algebra can be replaced by a weakly equivalent strict algebra which is a purely algebraic structure on a space. The equivalence between the homotopy categories of loop spaces and topological groups is a special instance of this result. 

In this paper we will introduce the notion of a finite product sketch which is a useful generalization of a multi-sorted algebraic theory. We will show that in the setting of finite product sketches we can still obtain results paralleling these of Badzioch and Bergner, although a rigidification of a homotopy algebra over a finite product sketch is given by a strict algebra over an associated simplicial multi-sorted algebraic theory. 
\end{abstract}


\end{frontmatter}

\section{\bf Introduction}

Let $\bo\Delta$ be the category whose objects are finite ordered sets $[n] = \{0, 1, \dots, n\}$, 
$n\geq 0$, and whose morphisms are non-decreasing functions. 
By $\bo\Delta^{op}$ we denote the opposite category and let $\bo I$ be the category with two objects $i_{0}, i_{1}$ and one non-identity 
morphism $\pi\colon i_{0}\to i_{1}$.  Consider the product category $\bo\Delta^{op}\times \bo I$.  It was shown by Prezma in \cite{prezma} that functors $X\colon \Delta^{op}\times \bo I\ra\spaces$ which send $([n],i_1)$ to a space which is homotopy equivalent to $X([1],i_1)^n$ and which send $([n],i_0)$ to a space which is homotopy equivalent to $X([n],i_1)\times X([0],i_0)$ for all $n$ provide a formalism for describing actions of loop spaces on topological spaces (see Example \ref{NLOOP EXAMPLE} and Example \ref{PREZMA EXAMPLE} for more detail) .  

One might consider the work of Badzioch \cite{badziochI},\cite{badziochIII} and Bergner \cite{Bergner} (see Theorem \ref{MULTISORTED THM}) to attempt to construct an "equivalent" algebraic structure to describe these actions, however the category $\bo\Delta^{op}\times \bo I$ is an example of a finite product sketch \ref{FINITE DEF} which is neither a multi-sorted algebraic theory \ref{ALGTH DEF} or a semi theory \ref{SEMITH DEF} (over a set of size 1) for which such a rigidification is known.  With this in mind, we consider more general finite product sketches.    

For $n\geq 1$ let $\CCC_{n}$ denote the category whose objects are natural numbers 
$0, 1, \dots, n$, and that has one non-identity morphism $p^{n}_{k}\colon 0\to k$ 
for each $k=1, \dots, n$. Also, by $\CCC_{0}$ we will denote the category
with only one object $0$ and the identity morphism only. 
Given a small category 
$\B,$   an $n$-fold \emph{cone} in $\B$ is a functor $\alpha\colon \CCC_{n}\to \B$  (this agrees with the traditional category theory definition of a cone over a discrete category with $n$ objects). 
It will be convenient,  given an $n$-fold cone $\alpha$,  to denote $n$ by $|\alpha|$, 
$\alpha(k)$ by $\alpha_{k}$ and $\alpha(p^{n}_{k})$ by $p^{\alpha}_{k}$. 
We will call the morphism $p^{\alpha}_{k}$ the \textit{$k$-th projection  of} $\alpha$. 

\begin{definition}
\label{FINITE DEF}
A \emph{finite product sketch} is a pair $(\B, \kappa)$ where $\B$ is a small category and 
$\kappa$ is a set of cones in $\B.$
\end{definition}

Sometimes,  when this will not lead to a confusion, we will write $\B$ to refer to the sketch $(\B, \kappa)$.

\begin{definition}
Let $\D$ be a category closed under finite products 
and let $(\B, \kappa)$ be a finite product sketch.  A  \emph{strict $(\B, \kappa)$-algebra} with values in $\D$ is 
a functor
$$A\colon \B\to \D$$
such that  for any  cone $\alpha\in \kappa$ the morphism  
$$\prod_{k=1}^{|\alpha|}A(p^{\alpha}_{k})\colon A(\alpha_{0}) \to \prod^{|\alpha|}_{k=1}A(\alpha_{k})$$
is an isomorphism in $\D$. 
For a cone $\alpha\colon \C_{0}\to \B$ this condition means that 
$\alpha_{0}$ is a terminal object in $\D$. A \textit{morphism} of strict algebras is 
a natural transformation of functors.
\end{definition} 

Finite product sketches and their strict algebras have long been present in categorical algebra 
as a formalism for describing algebraic structures. Giving  a strict $(\B, \kappa)$-algebra 
amounts to describing some algebraic object in $\D$ of the type determined by the sketch 
$(\B, \kappa)$. We illustrate this by a few examples.  

\begin{example}
Let $\B$ be the category consisting of two objects $b_{1}, b_{2}$, and three 
non-identity morphisms $\varphi_{1}, \varphi_{2}, \mu\colon b_{2}\to b_{1}$. Let 
$\alpha\colon \C_{2}\to \B$ be given by $p^{\alpha}_{k} = \varphi_{k}$.  
A strict algebra over the sketch $(\B, \{\alpha\})$ with values in the category 
of sets, $\Set$ is a functor  
$$A\colon \B \to \Set$$
such that $A(b_{2})\cong A(b_{1})\times A(b_{1})$ and 
$A(\varphi_{1}), A(\varphi_{2})\colon A(b_{2})\to A(b_{1})$ are the projection maps. 
In effect the category of strict $(\B, \{\alpha\})$-algebras is equivalent to the category 
whose objects are sets $Y$ equipped with a binary operation $\mu_{Y}\colon Y\times Y \to Y$, 
and whose morphisms are maps $f\colon Y\to Y'$ satisfying 
$$\mu_{Y'}(f(y_{1}), f(y_{2})) = f(\mu_{Y}(y_{1}, y_{2}))$$
It is not hard to modify this example to obtain sketches whose algebras are sets with a 
binary operation that satisfies some further conditions (is unital, associative, commutative, has 
inverse etc.).
\end{example}

\begin{example}
\label{WEAKEQ EXAMPLE}
Let $\B$ be a small category and let $\bo W\subseteq \B$ be a subcategory of 
$\B$. For each morphism $\varphi\in \bo W$ let $\alpha_{\varphi}\colon \C_{1}\to \B$
be the functor given by $p^{\alpha_{\varphi}}_{1} = \varphi$. Giving a strict algebra 
over the sketch $(\B, \{\alpha_{\varphi}\}_{\varphi\in \bo W})$ with values in 
a category $\D$ amounts to giving a functor $A\colon \B\to \D$ such that 
$A(\varphi)$ is an isomorphism for each $\varphi\in \bo W$. 
\end{example}

The next few examples are somewhat more complex. Their origin will be explained later in 
this Section.

\begin{example}[Segal's category of finite sets as a finite product sketch]
\label{GAMMA EXAMPLE}
Let $\bo \Gamma^{op}$ denote the category whose objects are finite sets $[n] = \{0, 1, \dots, n\}$, 
$n\geq 0$ and whose morphisms are maps of sets $\varphi\colon [n]\to [m]$ satisfying 
$\varphi(0)=0$. For $n\geq 1$ let $\alpha^{n}\colon \C_{n}\to \bo\Gamma^{op}$ be the functor 
such that $\alpha_{0} = [n]$, $\alpha_{k} = [1]$ for $k>0$ and $p^{\alpha^{n}}_{k}\colon [n]\to [1]$ 
is the map given by 
$$
p^{\alpha^{n}}_{k}(i) = 
\begin{cases}
1 & \text{for } i=k \\
0 & \text{otherwise} \\
\end{cases}
$$
Also, let $\alpha^{0}\colon \C_{0}\to \bo\Gamma^{op}$ be given by $\alpha^{0}_{0} = [0]$.
Giving a strict $(\bo\Gamma^{op}, \{\alpha^{n}\}_{n\geq 0})$-algebra
$A\colon \bo\Gamma^{op}\to \Set$ 
is equivalent to defining a structure of an abelian monoid on the set $A([1])$. 
\end{example}

\begin{example}
\label{NLOOP EXAMPLE}
As in the beginning of the introduction, let $\bo\Delta$ be the category whose objects are finite ordered sets $[n] = \{0, 1, \dots, n\}$, 
$n\geq 0$, and whose morphisms are non-decreasing functions. 
Let $\bo\Delta^{op}$ denote the opposite category. 
For $m\geq 1$ define a sketch $(\bo\Delta^{op}, \kappa^{m})$ as follows. 
We have
$$\kappa^{m} = \{\alpha^{k}\}_{k\geq 0}$$
where for $k< m$ the functor $\alpha^{k}\colon \C_{0} \to \bo\Delta^{op}$ is given by 
$\alpha^{k}_{0} = [k]$. For $k\geq m$ let $N_{k} = \binom{k}{m}$ and let $f_{k}$
be a bijection between the set of natural numbers $\{1, \dots, N_{k}\}$ and the set 
of strictly increasing functions $\varphi\colon [m]\to [k]$ satisfying $\varphi(0) = 0$ (we will see that the choice of $f_{k}$ is equivalent to choosing which projection morphism to call the $k$-th projection). 
Define $\alpha^{k}\colon \C_{N_{k}}\to \bo\Delta^{op}$ by setting 
$p^{\alpha^{k}}_{i} = f_{k}(i)$.  
For $m=1$  giving a strict $(\bo\Delta^{op}, \kappa^{m})$-algebra $A\colon \bo\Delta^{op}\to \Set$
amounts to defining a group structure on the set $X([1])$. For $m> 1$ giving a strict   
$(\bo\Delta^{op}, \kappa^{m})$-algebra $A\colon \bo\Delta^{op}\to \Set$ is equivalent 
to specifying an abelian group structure on $A([m])$ (see \cite{bousfield}).  
\end{example}

\begin{example}
\label{PREZMA EXAMPLE}

Let $\bo\Delta$ be defined as in Example \ref{NLOOP EXAMPLE}. For $n\geq 1$
and $1\leq k \leq n$ let $i^{n}_{k}\colon [1]\to [n]$ be the morphism in $\bo\Delta$
given by $i^{n}_{k}(0) = k-1$ and $i^{n}_{k}(1) = k$. Also, let 
$j^{n}_{1}, j^{n}_{2}\colon [0]\to [n]$ be given by $j^{n}_{1}(0) = 0$ and 
 $j^{n}_{2}(0) = n$.  

Let $\bo I$ be the category with two objects $i_{0}, i_{1}$ and one non-identity 
morphism $\pi\colon i_{0}\to i_{1}$.  Consider the product category $\bo\Delta^{op}\times \bo I$. 
For $n\geq 1$ let 
$$\alpha^{n}\colon \C_{n}\to \bo\Delta^{op}\times \bo I$$
be the cone such that   $p^{\alpha^{n}}_{k}$ is  the morphism 
$i^{n}_{k}\times \id_{i_{1}}\colon ([n], i_{1}) \to ([1], i_{1})$, and
let $\alpha^{0}$ be the $0$-fold cone given by $\alpha^{0}_{0} = ([0], i_{1})$. 
Also, for $n\geq 0$ let 
$$\beta^{n}, \gamma^{n} \colon \C_{2}\to \bo\Delta^{op}\times \bo I$$
be defined by $p^{\beta^{n}}_{1} = j^{n}_{1}\times \id_{i_0}$,  
$p^{\gamma^{n}}_{1} = j^{n}_{2}\times \id_{i_0}$, 
and $p^{\beta^{n}}_{2} = p^{\gamma^{n}}_{2} = \id_{n}\times \pi$. 
Define 
$$\kappa := \{\alpha^{n}\}_{n\geq 0} \cup \{\beta^{n}\}_{n\geq 0}\cup  \{\gamma^{n}\}_{n\geq 0}$$
Giving a strict $(\bo\Delta^{op}\times \bo I, \kappa)$-algebra 
$$X\colon \bo\Delta^{op}\times \bo I\to \Set$$
is equivalent to describing an associative monoid structure on the set $X([1], i_{1})$ and an action of 
this monoid on the set $X([0], i_{0})$. Giving a morphism of strict algebras $\varphi\colon X\to X'$ amounts to 
giving maps $\varphi_{1}\colon X([1], i_{1}) \to X'([1], i_{1})$ and 
$\varphi_{2}\colon X([0], i_{0}) \to X([0], i_{0})$ where $\varphi_{1}$ is a homomorphism of monoids 
and  $\varphi_{2}$ preserves the  action:
$$\varphi_{2}(mx) = \varphi_{1}(m)\varphi_{2}(x)$$
for all $x\in X([0], i_{0})$ and $m\in X([1], i_{1})$  (see \cite{prezma}). 
\end{example}

Finite product sketches have been shown to be very useful for describing various structures that appear in 
homotopy theory. In this context we need to replace  the notion of a strict algebra over 
a sketch by a more flexible notion of a homotopy algebra. Let $\spaces$ denote the category 
of simplicial sets. 

\begin{definition}
Let $(\B, \kappa)$ be a finite product sketch. A \emph{homotopy $(\B, \kappa)$-algebra}
is a functor
$$X\colon \B \to \spaces$$
such that for any  cone $\alpha\in \kappa$ the map 
$$\prod_{k=1}^{|\alpha|}X(p^{\alpha}_{k})\colon X(\alpha_{0}) \to \prod^{|\alpha|}_{k=1}X(\alpha_{k})$$
is a weak equivalence in $\spaces$.  
For a $0$-fold cone $\alpha$ this condition means that 
$X(\alpha_{0})\simeq \ast$.  A \textit{morphism} of homotopy algebras is a natural transformation (i.e. morphism in the category $\spaces^{\B}$).
\end{definition}

The  sketches described in Examples \ref{WEAKEQ EXAMPLE}--\ref{PREZMA EXAMPLE}
were all constructed because of the interest in the homotopy algebras they define. 
Homotopy algebras over the sketch $(\B, \{\alpha_{\varphi}\}_{\varphi\in \bo W})$ defined in Example \ref{WEAKEQ EXAMPLE}
are functors $\B\to \spaces$ that map morphisms of $\bo W$ to weak equivalences. 
Categories of such functors were studied e.g. by Dwyer and Kan in \cite{dwyer-kanI} who 
also gave several applications of such functors in homotopy theory. Homotopy 
algebras over the sketch $(\bo\Gamma^{op}, \{\alpha^{n}\}_{n\geq 0})$ defined in Example \ref{GAMMA EXAMPLE} are the special 
$\bo\Gamma$-spaces that were introduced by Segal in \cite{segal} to describe the structure 
of infinite loop spaces.  In \cite{bousfield} Bousfield showed that  the homotopy category of 
homotopy algebras over the sketch $(\bo\Delta^{op}, \kappa^{m})$ defined in Example 
\ref{NLOOP EXAMPLE} is equivalent to the homotopy category of $m$-fold loop spaces. Finally, homotopy 
algebras over the sketch $(\bo\Delta^{op}\times \bo I, \kappa)$ defined in Example \ref{PREZMA EXAMPLE} appear in the work of Prezma
\cite{prezma} as a formalism that describes ``homotopy actions'', i.e. actions of loop spaces on 
topological spaces.

Obviously any strict algebra over a sketch $(\B, \kappa)$ with values in $\spaces$ is  
also a homotopy $(\B, \kappa)$-algebra. A natural question is if the opposite statement 
also holds. More precisely, we will say that a morphism $\varphi\colon X\to X'$ of homotopy 
algebras is a weak equivalence if the map $\varphi_{b}\colon X(b)\to X'(b)$ is a weak equivalence 
for all  $b\in \B$.  We can ask if it is true that for any homotopy $(\B, \kappa)$-algebra $X$ there exists a strict 
$(\B, \kappa)$-algebra $X'$ such that $X'\simeq X$.  We will call this a \textit{rigidification problem} since 
it asks whether a homotopy (pseudo) structure can be replaced by an equivalent strict structure.

The examples of finite product sketches listed above show that  rigidification of homotopy algebras 
is not always possible.  Take e.g. the sketch $(\bo\Delta^{op}, \kappa^{2})$  described in Example
\ref{NLOOP EXAMPLE}.  Strict algebras over this sketch with values in $\spaces$ are simplicial 
abelian groups, while homotopy algebras correspond to double loop spaces. Since it is not true that every 
double loop space is  weakly equivalent to a simplicial abelian group (a consequence of \cite[Ch. 4, B.2.1]{Farjoun} by Dror Farjoun and the fact that not every double loop space is a product of Eilenberg-Mac Lane spaces), it is in general not possible to find 
a strict $(\bo\Delta^{op}, \kappa)$-algebra weakly equivalent to a given homotopy algebra. 

Badzioch \cite{badziochI} and Bergner \cite{Bergner} showed that homotopy 
algebras can be rigidified if the sketch $(\B, \kappa)$  is of a special form. In order to explain 
their results we will need a couple of definitions.

\begin{definition}
A \emph{product cone}  in a category $\B$ is a cone $\alpha\colon \CCC_{n}\to \B$ such that the categorical 
product $\prod_{k=1}^{n} \alpha(k)$ exists in $\B$ and that the map 
$$\prod_{k=1}^{n} \alpha(p^{n}_{k})\colon \alpha(0) \to \prod_{k=1}^{n} \alpha(k)$$
is an isomorphism in $\B$. 
\end{definition}

\begin{definition}[cf. {\cite[3.1]{Bergner}\cite[Section 6]{Fiore}}]
\label{ALGTH DEF}
Let $S$ be a set. An \emph{$S$-sorted algebraic theory} is a finite product sketch $(\T, \kappa)$
satisfying the following properties:
\begin{itemize}
\item [i)]  objects $t_{\underline s}\in \T$ are indexed by all $n$-tuples, 
${\underline s} = (s_{1}, \dots, s_{n})$ (with possible repetitions) where  and $s_{i}\in S$ and $n\geq 0$;
\item [ii)]  for an $n$-tuple ${\underline s} = (s_{1}, \dots, s_{n})$ the object $t_{\underline s}$ 
is a categorical product  of $t_{s_{1}}, \dots t_{s_{n}}$ (by abuse of notation we will denote 
by $t_{s_{1}}$ the object of $\T$ indexed by the $1$-tuple $(s_{1})$); 
\item [iii)] the set $\kappa$ consists of all product cones $\alpha^{\underline s}$ indexed by  
$n$-tuples ${\underline s} = (s_{1}, \dots, s_{n})$ with $n\geq 0$,
where $\alpha^{\underline s}_{0} = t_{\underline s}$ and 
$\alpha^{\underline s}_{k} = t_{s_{k}}$ for $k=1, \dots, n$.
\end{itemize}
\end{definition}

We will call $S$  the \textit{set of sorts} for $\T$. We will also call a sketch $\T$ a multi-sorted 
algebraic theory if $\T$ is an $S$-sorted algebraic theory for some set $S$. Notice that 
a strict algebra over a multi-sorted algebraic theory $\T$ with values in $\spaces$ is just a product 
preserving functor $\T \to \spaces$, and a homotopy $\T$-algebra is a functor that preserves 
products up to a weak equivalence. 

\begin{remark}
It can be noticed that algebraic theories, introduced by Lawvere in \cite{lawvere} are $S$-sorted algebraic theories over a set $S$ of size 1 and thus also finite product sketches.
\end{remark}

The result of Bergner and Badzioch can be now stated as follows:

\begin{theorem}[\cite{badziochI},  \cite{Bergner}] 
\label{MULTISORTED THM}
If $\T$ is a multi-sorted algebraic theory then any homotopy 
$\T$-algebra is weakly equivalent to a strict $\T$-algebra. 
\end{theorem}

As we have already noticed this Theorem cannot be directly extended to 
arbitrary finite product sketches.  Our main goal, however, is to show that a variant 
of this rigidification result still holds. Given a finite product sketch $\B$ we can consider
the homotopy category of homotopy algebras over $\B$, that is the category obtained 
by taking the category of all homotopy $\B$-algebras and inverting weak equivalences. 
We will show that the following holds.

\begin{theorem}
 \label{MAINI COR}
For any finite product sketch $\B$ there exists a simplicial multi-sorted 
algebraic theory $\FB$ such that the homotopy
category of homotopy $\B$-algebras is equivalent to the homotopy 
category of strict $\FB$-algebras. 
\end{theorem}

Thus, the homotopy structure described by any finite product sketch is equivalent 
to some algebraic structure, but in general the sketches describing these 
two structures will be different and the associated algebras will be simplicial, i.e. bisimplicial sets. 

As an application of  Theorem \ref{MAINI COR} we  partially resolve another 
natural question related to homotopy algebras over finite product sketches. Namely, assume that 
$$G\colon (\B_{1}, \kappa_{1}) \to (\B_{2}, \kappa_{2})$$ 
is a morphism of sketches.  That is, $G$ is a functor such that for any $\alpha\in \kappa_{1}$
the cone $G\alpha$ is in $\kappa_{2}$.
Obviously if $X\colon (\B_{2}, \kappa_{2}) \to \spaces$
is a homotopy $(\B_{2}, \kappa_{2})$-algebra then $G^{\ast}X := XG$ is a homotopy 
$(\B_{1}, \kappa_{1})$-algebra. One can ask what conditions on $G$ guarantee that 
the functor $G^{\ast}$ is an equivalence of the homotopy categories of homotopy algebras. 
This can be answered as follows. Since the passage from a finite product sketch $(\B, \kappa)$
to its associated multi-sorted algebraic theory $\FB$ is natural in $(\B, \kappa)$ a morphism 
of sketches $G\colon (\B_{1}, \kappa_{1}) \to (\B_{2}, \kappa_{2})$ yields a functor 
of the simplicial algebraic theories $G\colon \FBO \lra \FBT$ (provided $G$ is injective). We will prove the following Theorem in Section \ref{bigpf}. 

\begin{theorem}
\label{MAINII}
Consider a morphism  of finite product sketches 
$$G\colon (\bo{B_1}, \kappa_{1}) \lra (\bo{B_2}, \kappa_{2})$$
with the property that $G$ is an injection on the sets of objects of $\bo{B_1}$ and $\bo{B_2}$.  Each of the following holds:
\begin{itemize}
\item[i)]  If the induced functor $$G\colon \FBO \lra \FBT$$ is a weak r-equivalence of simplicial categories, i.e. $G$ has the property that for any 
two objects $B_1$ and $B_2$ in $\FBO$, the induced map $ \hom(B_1,B_2)\ra  \hom(GB_1,GB_2)$ is a weak equivalence and every object in $\FBT$ is a retract of an object in the image of $G$, then $G$ induces an equivalence of the homotopy categories of homotopy algebras.
\item[ii)]
If we assume that $G$ is also surjective on the sets of objects of $\bo{B_1}$ and $\bo{B_2}$ and a bijection on the sets of cones, then $G$ induces an equivalence of the homotopy category of homotopy algebras iff the induced functor  $$G\colon \FBO \lra \FBT$$ is a weak equivalence of simplicial categories. 
\end{itemize}
\end{theorem}

This  fact parallels  the result of Dwyer and Kan \cite[2.5]{dwyer-kanI}
who proved an analogous statement for functors between sketches of the form described in 
Example \ref{WEAKEQ EXAMPLE}, and a Theorem of Badzioch \cite[1.6]{badziochIII}, that gives
a similar criterion for functors between single-sorted semi-theories, i.e. finite product sketches of a specific 
type (cf. Definition \ref{SEMITH DEF}).

%%%%%%%%%%%%%%%%%%%%%%%%%%%%%%%%%%%%%%%%%%%%%%%%%%%%%%%
%
%          2. Organization of Paper
%
%%%%%%%%%%%%%%%%%%%%%%%%%%%%%%%%%%%%%%%%%%%%%%%%%%%%%%%

\section{\bf Organization of Paper}

We start in Section  \ref{Multi-Sorted Semi-Theories} by considering a specific type of 
finite product sketch which we refer to as a multi-sorted
semi-theory, and in Sections \ref{SET UP} through \ref{PF LEMMA}
we show that we can rigidify homotopy algebras over multi-sorted
semi-theories by constructing an associated multi-sorted algebraic theory.
In particular in Section \ref{SET UP} we show that the setup, paralleling 
\cite{badziochIII}, can be used to  define  model category structures for the categories 
of homotopy algebras as well as strict algebras over a multi-sorted semi-theory.  
In Section \ref{COMPLETION} we show that for any  multi-sorted semi-theory 
$\C$ we can construct a multi-sorted algebraic theory $\CC$ without changing the 
category of strict algebras. We also give an explicit,  combinatorial constriction of $\CC$ 
in the case when $\C$ is a free semi-theory. 
In Section \ref{INITIAL} the initial  semi-theory $\Prj$ is introduced and it is shown that it can be 
used to detect weak equivalences  in the category
of homotopy algebras over a multi-sorted semi-theory.
In Section \ref{PF LEMMA} we  complete the argument showing that homotopy algebras over
a multi-sorted semi-theory can be rigidified as strict algebras
over a certain multi-sorted algebraic theory. In Section \ref{PF THM} we prove 
the variant of Theorem  \ref{MAINII} for  multi-sorted semi-theories.
Given an arbitrary finite product sketch we show in Section \ref{assmsst} that we can construct
an associated multi-sorted semi-theory for which the homotopy category of homotopy algebras 
is equivalent to the one defined by the original sketch. Using this result in Section \ref{bigpf} 
we prove Theorems \ref{MAINI COR} and \ref{MAINII}.

%%%%%%%%%%%%%%%%%%%%%%%%%%%%%%%%%%%%%%%%%%%%%%%%%%%%%%%
%
%          3. Multi-Sorted Semi-Theories
%
%%%%%%%%%%%%%%%%%%%%%%%%%%%%%%%%%%%%%%%%%%%%%%%%%%%%%%

\section{\bf Multi-Sorted Semi-Theories}
\label{Multi-Sorted Semi-Theories}

While our ultimate goal in this paper is to show that we can  rigidify algebras over arbitrary finite product
sketches in the sense of Theorem \ref{MAINI COR}, we will first  consider the rigidification problem for a specific type of finite product 
sketch, which we will call  multi-sorted semi-theories.

\begin{definition}
\label{SEMITH DEF}
Let $S$ be a set. An \emph{$S$-sorted semi-theory} is a finite product sketch $(\C, \kappa)$
satisfying the following properties:
\begin{itemize}
\item [i)]  Objects $c_{\underline{s}}\in \C$ are indexed by $n$-tuples
$\underline{s} = (s_{1}, \dots, s_{n})$ (with possible repetitions) where $s_{i}\in S$ and $n\geq 0$.
By abuse of notation for $s_{1}\in S$ we will write $c_{s_{1}}$ to denote the object of $\C$ indexed by the
$1$-tuple $(s_{1})$.  
\item [ii)] For any $n$-tuple $\underline{s} = (s_{1}, \dots, s_{n})$ 
there is a unique $n$-fold cone $\alpha^{\underline{s}}\in \kappa$
such that  $\alpha^{\underline{s}}_{0} = c_{\underline{s}}$,  
$\alpha^{\underline{s}}_{k} = c_{s_{k}}$ for $k=1, \dots, n$. 
Moreover, every cone in $\kappa$ is of such form. 
\end{itemize}
\end{definition}

Given an $S$-sorted semi-theory $\C$ for simplicity we will denote the $k$-th  projection in the cone 
$\alpha^{\underline{s}}$ by $p^{\underline{s}}_{k}$ instead of $p^{\alpha^{\underline{s}}}_{k}$. 
We will say that $\C$ is a multi-sorted semi-theory if $\C$ is $S$-sorted for 
some set $S$. 

Notice that the  definition of an $S$-sorted semi-theory parallels that of an $S$-sorted algebraic theory 
(\ref{ALGTH DEF}).  The only difference is that we do not assume that the cones in $\kappa$ 
are product cones.  Our first goal will be to show that  variants of  Theorems 
\ref{MAINI COR} and \ref{MAINII}  hold for multi-sorted semi-theories.

\begin{theorem}
\label{mainIII}
 For any $S$-sorted semi-theory $\CATC$ there exists an $S$-sorted
algebraic theory $\bFC$ such that the homotopy category of homotopy $\CATC$-algebras is equivalent to the homotopy category of strict
$\bFC$-algebras.  Moreover, the construction of  $\bFC$ is functorial in $\CATC$.
\end{theorem}

\begin{theorem}
\label{mainIV}
Let $\C$, $\C'$ be $S$-sorted and $S'$-sorted semi-theories respectively, and let $G\colon \C \to \C'$ be a morphism of finite 
product sketches.  We have the following
\begin{itemize}
\item[i)] If the induced functor $$G \colon \bFC \ra {\bFC'}$$ between the associated multi-sorted theories is a weak r-equivalence, then $G$ induces an equivalence of the homotopy categories of homotopy algebras.
\item[ii)] If in addition there exists a bijection of sets $\varphi \colon S\ra S'$ and $G$ preserves sorts, i.e. $G(c_{\underline s}) = c'_{ \varphi(\underline s)}$ for any $n$-tuple $\underline s$ in $S$
then the  functor $G$ induces an equivalence of the homotopy categories of homotopy algebras
iff the induced functor 
$$G \colon \bFC \ra {\bFC'}$$ 
between the associated multi-sorted theories is a weak equivalence of simplicial categories.
\end{itemize}
\end{theorem}

After proving these facts in Sections \ref{PF LEMMA} and \ref{PF THM} we will show how they can be used to obtain Theorems \ref{MAINI COR} and \ref{MAINII} 
in their whole generality.

From now until Section \ref{assmsst} it  will be convenient to fix the set $S$ and assume that all 
multi-sorted semi-theories are $S$-sorted. Consequently, all morphisms of semi-theories
will be assumed to preserve sorts and projections. The resulting category of  $S$-sorted semi-theories 
will be denoted by $\bo{SemiTh}_{S}$. Notice that the category of $S$-sorted algebraic theories 
$\bo{AlgTh}_{S}$ is a full subcategory of $\bo{SemiTh}_{S}$.

%%%%%%%%%%%%%%%%%%%%%%%%%%%%%%%%%%%%%%%%%%%%%%%
%%%%%%%%%%%%%%%%%%%%%%%%%%%%%%%%%%%%%%%%%%%%%%%
%
%          4. Model Category structures
%
%%%%%%%%%%%%%%%%%%%%%%%%%%%%%%%%%%%%%%%%%%%%%%%
%%%%%%%%%%%%%%%%%%%%%%%%%%%%%%%%%%%%%%%%%%%%%%%

\section{\bf Model Categories of Strict and Homotopy Algebras}
\label{SET UP}

Our basic strategy for proving Theorem \ref{mainIII} will be to rephrase 
it in terms model categories and prove that the relevant model categories are 
Quillen equivalent. With this in mind our first task will be to introduce model 
category structures reflecting the homotopy theories of strict and homotopy algebras. 
Much of what is discussed here parallels the setup of  \cite{badziochI}, \cite{badziochIII}, and \cite{Bergner} 
so the presentation will be brief.

{\noindent \bf The model category of strict algebras.\ }

Let $\C$ be a multi-sorted semi-theory and let $\algc$ denote the full subcategory
of $\spaces^{\C},$ whose objects are strict algebras over $\C.$
We would like to get a model category structure on $\algc$ where $\C$ is 
a multi-sorted semi-theory.  To do this let us first consider the category 
$\algt$ when $\T$ is an $S$-sorted algebraic theory.  

We will use the following fact due to Kan.

\begin{theorem}[{\cite[11.3.2]{hirschhorn}}]
\label{hirsch}

Let $\M$ be a cofibrantly generated model category with a set
of generating cofibrations $I$ and generating acyclic cofibrations $J$.  
Let $\bo N$ be a category which has all small limits and small colimits
and for which there exists
a pair of adjoint functors:

$$F \colon \M \rightleftarrows {\bo N} \colon U$$
with $FI=\{Fu | u \in I\}$ and $FJ=\{Fv | v \in J \}$
and 
\begin{itemize}
\item[i)] $FI$ and $FJ$ permit the small object argument 

\item[ii)] $U$ takes colimits of pushouts along maps in $FJ$ to weak equivalences

\end{itemize}

\noindent Then there is a cofibrantly generated model category structure on $\bo N$
for which $FI$ is a set of generating cofibrations, $FJ$ is a set of generating
acyclic cofibrations and the set of weak equivalences is the set of maps which 
$U$ sends to weak equivalences in $\M$.  Furthermore, with respect to this model
category structure, $(F,U)$ is a Quillen pair.

\end{theorem}

\begin{proposition}
\label{ALGTHMODSTR PROP}
Let $\T$ be an $S$-sorted algebraic theory. 
The category of strict $\T$-algebras  $\algt$ admits a model category structure 
defined by  the following classes of morphisms:

\begin{itemize}
\item[i)] weak equivalences are  objectwise weak equivalences;
\item[ii)] fibrations are objectwise fibrations;
\item[iii)] cofibrations are morphisms which have the left lifting property with respect
to acyclic fibrations.
\end{itemize}

\end{proposition}

\begin{proof}
In \cite[\S 4]{Bergner} Bergner showed that  
for each $s\in S$  the evaluation functor
$$U_{s}\colon \algt \lra \spaces, \ \ \ \ U_{s}(X)  = X(t_{s})$$
has a left adjoint $F_{s} \colon \spaces \lra \algt$.

Consider the set $S$ as a category with identity morphisms only and let $\spaces^{S}$
be the category of functors $S \to \spaces$. Notice that objects of $\spaces^{S}$ are just 
assignments that associate to each element $s\in S$ a space $Y_{s}$. The category 
$\spaces^{S}$ has a model structure with fibrations, cofibrations, and weak equivalences 
defined objectwise.  The  forgetful functor  
$$U\colon \algt \to \spaces^{S}, \ \ \ U(X)(s) = X(t_{s})$$
has a left adjoint $F$ defined for $Y\in \spaces^{S}$ by 
$$F(Y) = \coprod_{s\in S} F_{s}(Y(s))$$
where the 
coproduct is taken in $\algt$. 

Note that the model category  $\spaces^S$ is cofibrantly generated with the set of generating cofibrations
$$I = \{ \partial{\Delta}[n]_{s_i} \ra \Delta [n]_{s_i} | n\geq 0, s_i \in S\}$$
and a set of generating acyclic cofibrations
$$J = \{V[n,k]_{s_i} \ra \Delta [n]_{s_i} | n\geq 1, 0\leq k \leq n, s_i \in S\}.$$
Here $\partial{\Delta}[n]_{s_i} \in \spaces^S$ is defined by
$\partial{\Delta}[n]_{s_i}(s_j)= \partial{\Delta}[n]$ for $s_j=s_i$ and $\partial{\Delta}[n]_{s_i}(s_j)$ is 
the empty set if $s_j \neq s_i$.  The objects $\Delta [n]_{s_i}$, and $V[n,k]_{s_i}$ in $\spaces^S$ are  
defined in a similar matter. Applying Theorem \ref{hirsch} to the adjoint pair $(F, U)$ 
we get a model category structure on $\algt$ as described in the statement. 
\qed
\end{proof}

Proposition \ref{ALGTHMODSTR PROP} can be easily generalized to the case of arbitrary 
multi-sorted semi-theories.

\begin{corollary}
Let $\C$ be an $S$-sorted semi-theory. 
The category of strict $\C$-algebras,  $\algc$ admits a model category structure 
defined by  the following classes of morphisms:

\begin{itemize}
\item[i)] weak equivalences are  objectwise weak equivalences;
\item[ii)] fibrations are objectwise fibrations;
\item[iii)] cofibrations are morphisms which have the left lifting property with respect
to acyclic fibrations.
\end{itemize}

\end{corollary}

\begin{proof}
By \cite[Chapter 4, Theorem 3.6]{Barr} we get that any $S$-sorted semi-theory  $\C$
has an associated $S$-sorted algebraic theory $\CC$ with the property that
$\algc$ and $\algcc$ are equivalent categories, so the statement follows directly from 
Proposition \ref{ALGTHMODSTR PROP}.
\qed
\end{proof}

\vspace{10mm}
{\noindent \bf The model category of homotopy algebras \ }

 Let $\C$ be a multi-sorted semi-theory. Our next goal is to describe a model 
structure that reflects the homotopy theory of homotopy $\C$-algebras. We cannot
do this arguing along the same lines as in the case of strict algebras  since the full subcategory 
of $\spaces^{\C}$ that consists  of homotopy algebras is not closed under colimits. 
Instead, we will obtain the desired model category by localizing the functor category  
$\spaces^{\C}$.  Also, since we will eventually need  a model category structure of 
homotopy algebras over an arbitrary finite product sketch (and not just for homotopy 
algebras over a multi-sorted semi-theory) we will work here in this more general setting.  

For a small category $\C$ the functor category $\spaces^{\C}$ can be  equipped with two 
different model category structures
 which we will denote by $\spaces^{\C}_{fib}$ and $\spaces^{\C}_{cof}$.  
In both of these model categories  
weak equivalences are objectwise weak equivalences.  In $\spaces^{\C}_{fib}$ 
fibrations are   the objectwise fibrations, and in $\spaces^{\C}_{cof}$, cofibrations are
the objectwise cofibrations.  In each case the third class of morphisms is determined 
by the lifting properties of model categories.  
By \cite[ VIII.1.4]{goerss} and \cite[IX.5.1]{goerss}  both $\spaces^{\C}_{fib}$ and $\spaces^{\C}_{cof}$ are simplicial 
model categories with the following simplicial structure:
for $X \in \spaces^{\C}$ and  $K \in \spaces$ the functor $X\otimes K \in \spaces^{\C}$ is given by 
$$(X \otimes K)(c) = X(c) \times K$$
for all $c\in \C$.  For $X, Y\in\spaces^{\C}$ by $\map_{\C}(X, Y)$ we will denote the 
associated simplicial mapping complex.

Let $(\C, \kappa) $ be a finite product sketch. We 
will consider an additional model category structure on $\spaces^{\C}$ denoted $\bo L \spaces^{\C}$. 
This model category is obtained as follows. For $c\in \C$ let $\C_{c}\in \spaces^{\C}$ denote the 
functor corepresented by $c$:  
$$\C_{c}(d) := \Hom_{\C}(c, d).$$
Given an $n$-fold cone $\alpha\in \kappa$ consider the morphism 
\begin{equation}
\label{LOC MAPS}
p^{\alpha\ast} := \coprod_{k=1}^{n} (p^{\alpha}_{k})^{\ast} \colon 
\coprod_{k=1}^{n} \C_{\alpha_k}\lra \C_{\alpha_{0}}.
\end{equation}

The model category $\bo L \spaces^{\C}$ is the left Bousfield localization of $\spaces^{\C}_{fib}$ with respect to 
the set $P = \{p^{\alpha \ast}\}_{\alpha\in \kappa}$. This localization exists by 
\cite[4.1.1]{hirschhorn} and by the fact that $\spaces^{\C}_{fib}$ is a left proper cellular model category.
The model category structure on $\bo L \spaces^{\C}$, can be described 
explicitly as follows:

\begin{itemize}

\item [i)] If  $X$ and $Y$ are cofibrant objects in $\spaces^{\C}_{fib}$ then 
a map $f\colon X\ra Y$ is a weak  equivalence  in $\bo L \spaces^{\C}$ 
if for every homotopy algebra $Z,$ fibrant in $\spaces^{\C}_{fib}$
the induced map of homotopy function complexes, 
$$f^\ast\colon\map_\C(Y,Z)\lra\map_\C(X,Z)$$
is a weak equivalence of simplicial sets. 
If $X$ and $Y$ are not cofibrant then the map $f$
is a weak  equivalence  in $\bo L \spaces^{\C}$  
if the induced map  $f'\colon X'\ra Y'$ between cofibrant replacements 
of $X$ and $Y$ is one.  We will call such map $f$  a local equivalence to distinguish it from 
and objectwise weak equivalence.

\item[ii)] Cofibrations in $\bo L \spaces^{\C}$ are the same as cofibrations in $\spaces^{\C}_{fib}$, 
and fibrations in $\bo L \spaces^{\C}$ are morphisms with the right lifting property with 
respect to local equivalences which are also cofibrations.

\item [iii)] An object $X \in \bo L \spaces^{\C}$ is fibrant iff $X$ is a homotopy $(\C, \kappa)$-algebra 
and $X$ is fibrant in $\spaces^{\C}_{fib}$.

\item[iv)] If $f\colon X\to Y$ is a map of homotopy $(\C, \kappa)$-algebras then $f$ is a local 
equivalence iff $f$ is an objectwise weak equivalence. 

\end{itemize}

\begin{note}
Later on we will frequently use the following, equivalent  description of local equivalences 
that can be obtained using arguments paralleling these given in \cite[Section 5]{badziochI}: 
a map $f\colon X\to Y$ is a local equivalence if for any homotopy $(\C, \kappa)$-algebra $Z$ 
which is a fibrant object in  $\spaces^{\C}_{cof}$ the induced map of simplicial function complexes
$$f^{\ast}\colon \map_{\C}(Y, Z)\to \map_{\C}(X, Z)$$
is a weak equivalence.
\end{note}

Combining properties iii) and and iv)  we obtain the following Proposition.

\begin{proposition}
Let $(\C, \kappa)$ be a finite product sketch. 
 The homotopy category of $\lsc$ is equivalent to the category obtained be taking the full subcategory of $\spaces^{\C}$
spanned by homotopy $(\C, \kappa)$-algebras and inverting all objectwise weak equivalences.
\end{proposition}
In other words, $\bo L \spaces^{\C}$ is a model category that describes the homotopy theory of 
homotopy $(\C, \kappa)$-algebras.

\vspace{10mm}
{\noindent \bf  The Quillen Pair between $\algc$ and $\lsc$ \ }

Assume now that $\C$ is a multi-sorted semi-theory and consider the inclusion functor 
$$J_{\C}\colon \algc \lra \spaces^{\C}.$$
Our next goal will be to show that $J_{\C}$ is the right adjoint  in a Quillen pair of functors between 
the model categories $\algc$ and $\bo L \spaces^{\C}$. If for some 
semi-theory  $\C$ we can show that  this Quillen pair is a Quillen equivalence we 
will obtain that the rigidification problem can be solved for homotopy algebras over $\C$:
any homotopy  $\C$-algebra is weakly equivalent to a strict $\C$-algebra.

The existence of a left adjoint to the functor $J_{\C}$ can be demonstrated using the approach used by 
Bergner in \cite{Bergner}.

\begin{definition}
 \cite[5.5]{Bergner}
Let $\D$ be a small category and let $P$ be a set of morphisms in $\spaces^{\D}$.  
An object $Y$ in $\spaces^{\D}$ is {\em strictly local with respect to $P$}
if for every $(f\colon A \ra B) \in P$, the induced map of the  simplicial function complexes 
$$f^{\ast} \colon \map(B,Y) \lra \map(A,Y)$$
is an isomorphism of simplicial sets.
\end{definition}

\begin{lemma}
 \cite[5.6]{Bergner}
\label{forget}
For a small category $\D$ and a set of morphisms $P$ in $\spaces^{\D}$ let
$\bo{Alg}^{(\D, P)}$ denote the full subcategory  of $\spaces^{D}$ whose objects are
strictly local diagrams.  The inclusion functor 
$$\bo{Alg}^{(\D, P)}\to \spaces^{\D}$$
has a left adjoint.
\end{lemma}

Using this Lemma we obtain the following Proposition.

\begin{proposition}
Let $\C$ be a multi-sorted semi-theory. There exists a functor 
$$K_{\C}\colon\spaces^{\C}\lra\alg{\C}$$
left adjoint to $J_{\C}$.  Furthermore, the pair $(K_{\C},J_{\C})$ is a Quillen pair
between the model categories $\spaces^{\C}_{fib}$ and $\algc$.
\end{proposition}

\begin{proof}
It is enough to notice that a strict $\C$-algebra is a diagram $X\colon \C \ra \spaces$ which is 
strictly local with respect to the set  $P=\{p^{\alpha \ast}\}_{\alpha\in \kappa}$ where 
$p^{\alpha \ast}$ is the map defined as in (\ref{LOC MAPS}).  To see that we have a Quillen pair,
notice that fibrations and weak equivalences in $\algc$ are computed objectwise, thus $J_{\C}$
preserves both.
\qed
\end{proof}

Next, we want show that  $(K_{\C},J_{\C})$ is still a  Quillen pair after we localize $\spaces^{\C}_{fib}$.  
This is a consequence of the following fact.

\begin{lemma}
\label{qp}
Let $\M$ and $\bo N$ be model categories and ${\bo L} \M$ be a left Bousfield localization
of $\M$. If 
$$F\colon \M \leftrightarrows {\bo N} \colon G$$ 
is a Quillen pair such  that for
any fibrant object $X \in \bo N$ the object $G(X)$ is fibrant in ${\bo L}\M$, then 
$$F\colon {\bo L}\M\leftrightarrows {\bo N} \colon G$$ is also a Quillen pair. 
\end{lemma}

\begin{proof}
By \cite[A.2]{Dugger} we only need to show that $G\colon {\bo N} \to {\bo L}\M$ preserves all fibrations between fibrant objects and preserves all acyclic fibrations. Assume then that $f\colon X \lra Y$ is a fibration between 
fibrant objects in $\bo N$. 
By assumption  $G(X)$ and $G(Y)$ are fibrant in 
${\bo L}\M$. Also, since $G\colon {\bo N}\to \M$  is a right Quillen functor the morphism  $G(f)$ is a fibration
in $\M$.  
Using the model category structure of ${\bo L}\M$ we can decompose $G(f)$ 
so that $$G(f)=G(X)\mapright{20}{\varphi} Z \mapright{20}{\psi} G(Y)$$
 where $\psi$ is a fibration in ${\bo L}\M$ and
$\varphi$ is an acyclic cofibration in ${\bo L}\M$. By \cite[3.3.14]{hirschhorn} we get  that $Z$ must be a local object, which by  \cite[3.2.13]{hirschhorn}  gives that $\varphi$ is  a weak equivalence in $\M$. 
Therefore by \cite[3.3.15]{hirschhorn}, $G(f)$ must  be a fibration in ${\bo L}\M$.

If $f\colon X \lra Y$ is an acyclic fibration in $\bo N$ then $G(f)$ is an acyclic
fibration in $\M$ since $G$ is a right Quillen functor.  It remains  to notice  that acyclic fibrations 
in ${\bo L} \M$ are the same as acyclic fibrations in $\M$ \cite[3.3.3]{hirschhorn}.
\qed  
\end{proof}

\begin{proposition}
\label{KC JC}
For any multi-sorted semi-theory
 the adjoint pair of functors $(K_{\C},J_{\C})$ is a Quillen 
  pair between the model categories $\lsc$ and $\alg{\C}$.
\end{proposition}

\begin{proof}
By Lemma \ref{qp} we  only need to show that for any fibrant object $X \in \algc$ the object 
$J_{\C}(X)$ is fibrant  in $\lsc$. This is obvious since fibrant objects in $\lsc$ are  
homotopy algebras, fibrant in $\spaces^{\C}_{fib}$.
\qed
\end{proof}

The next Lemma gives
a way of verifying that $(K_{\C},J_{\C})$ is a Quillen equivalence for a given multi-sorted semi-theory $\C$.  For
$X\in \spaces^{\C}$ let $$\eta_{X} \colon X \ra J_{\C}K_{\C}X$$
be the  unit of adjunction of $(K_{\C},J_{\C})$.

\begin{lemma}
\label{compare}
Let $\C$ be a multi-sorted semi-theory. 
 If the map $\eta_{\C_c}$ is a local equivalence for all $c \in \C$ then the Quillen pair $\KJ$  is a Quillen equivalence
of the model categories $\lsc$ and $\alg{\C}$.
\end{lemma}

The proof of Lemma \ref{compare} will use a couple of auxiliary facts. 

First,  for a simplicial model category  $\M$ let $s\M$ be the category of simplicial objects in 
$\M$, i.e. the category of functors  ${\bo\Delta}^{op}\to \M$. 
We have the geometric realization functor
$$|- | \colon s\M \lra \M$$
such that for $X_{\bullet}\in s\M$ the object  $|X_{\bullet}|\in \M$ is the coequalizer of the diagram:
$$\coprod_{\phi \colon [n] \ra [m]} X_m \otimes \Delta [n] \rightrightarrows \coprod_{[n]} 
X_n \otimes \Delta [n] .$$

\begin{lemma}
\label{badi}
 Let $\C$ be a small category and $\lsc$ be the left Bousfield localization of $\spaces^{\C}_{fib}$ with respect
to a set of maps $P$.  Assume that we have a Quillen pair 
$$K \colon \lsc \leftrightarrows \M \colon J$$  where $\M$ is assumed to be a simplicial model category and the following hold

\begin{itemize}
 \item [i)]for $X_{\bullet} \in s \spaces^{\C}$ we have  $|JKX_{\bullet} | \simeq JK |X_{\bullet}|$
 \item [ii)]$J$ commutes with filtered colimits
 \item [iii)]$J(f)$ is a local equivalence iff $f$ is a weak equivalence
 \item [iv)] the unit $\eta_{Y}$ of the adjunction is a local equivalence for $Y=\coprod_{i=1}^{m} \C_{c_i}$
where $\{c_i\}_{i=1}^m$ is any finite set of objects in $\C$
\end{itemize}
 Then $(K,J)$ is a Quillen equivalence.

\end{lemma}

\begin{proof}
We need to show that for any cofibrant object  $X\in \lsc$ and any fibrant object $Y\in \M$
a morphism $f\colon K(X) \to Y$ is a weak equivalence in $\M$ if and only if its transpose 
$f^{\sharp}\colon X \to J(Y)$ is a weak equivalence in $\lsc$. Recall that   $f^{\sharp}$ is 
given by the composition $f^{\sharp} = J(f)\eta_{X}$ where $\eta_{X}$ is the unit of 
adjunction. Since by assumption $J(f)$ is a weak equivalence if and only if $f$ is one, 
it will suffice to show that $\eta_{X}$ is a weak equivalence in $\lsc$ for all cofibrant 
objects $X\in \lsc$. 
 
Assume now that $X, \hat{X}\in \lsc$ are cofibrant objects, that $\eta_{\hat X}$ is a weak equivalence 
and that we also have a weak equivalence $g\colon \hat X \to X$. We claim that in such a case 
$\eta_{X}$ is also a weak equivalence. Indeed, we have a commutative diagram. 
\begin{equation*}
\begin{tikzpicture}[baseline=(current bounding box.center)]
\matrix (m) 
[matrix of math nodes, row sep=3em, column sep=3em, text height=1.5ex, text depth=0.25ex]
{
\hat X & JK(\hat X) \\
X & JK(X) \\ 
};
\path[->, thick, font=\scriptsize]
(m-1-1)
edge node[anchor=south] {$\eta_{\hat X}$} node[anchor=north] {$\simeq$} (m-1-2)
edge node[anchor=east] {$g$} node[anchor=west] {$\simeq$} (m-2-1)
(m-1-2)
edge node[anchor=west] {$JK(g)$} (m-2-2)
(m-2-1) 
edge node[below] {$\eta_{X}$} (m-2-2)
; 
\end{tikzpicture}
\end{equation*}
It suffices to show that $JK(g)$ is a weak equivalence in $\lsc$. 
Since $g$ is weak equivalence of cofibrant objects and $K$ is a left Quillen functor by 
\cite[Lemma 9.9]{dwyerspalinski} we obtain that the morphism $K(g)$ is a weak equivalence. Therefore 
by assumption $JK(g)$ is a weak equivalence. 

The above argument shows that in order to complete the proof it suffices to show that for any 
$X\in \lsc$ there exists  a cofibrant object $\hat X$ weakly equivalent to $X$ and such that 
$\eta_{\hat X}$ is a weak equivalence.   In order to get an  appropriate choice of $\hat X$
we can use a construction given by Badzioch. In  \cite[\S 3]{badziochI} he showed that to any 
$X \in \lsc$ we can associate a cofibrant replacement $\hat X$ ($\hat X = |FU_{\bullet} X|$ in the 
notation of \cite{badziochI}) such that: 
\begin{itemize}
\item[1)] $\hat X$ is obtained as the geometrical realization of a certain bisimplicial object 
$\hat X_{\bullet\bullet}$ in $\spaces^{\C}$;
\item[2)] in each bisimplicial degree the functor $\hat X_{m, n}\in \spaces^{\C}$ is given by 
a (possibly infinite) coproduct of corepresented functors:
$$\hat X_{m, n} = \coprod_{i\in I} \C_{c_i}.$$
\end{itemize}
Since by assumption the functor $JK$ commutes with geometric realization, in order to see that $\eta_{\hat X}$ is a weak equivalence for all $\hat X$ it will suffice 
to show that $\eta_{Z}$ is a weak equivalence for any $Z =  \coprod_{i\in I} \C_{c_i}$.
Notice that if $I$ is a finite set then this holds by assumption. Assume then that $I$ is infinite, and 
let ${\bo P}_I $ denote the category of all finite subsets of $I$ with morphisms given by 
inclusions of sets.  We have a functor:
$$\tilde{Z} \colon {\bo P}_I \lra \lsc, \ \ \ \ \ \ \tilde{Z}(A)  := \coprod_{i\in A}\C_{c_i}.$$ 
Notice that  $\colim[{\bo P}_{I}] \tilde{Z} =Z $. Also,  $\colim[{\bo P}_{I}] JK\tilde{Z} =JK(Z)$ 
since by assumption $J$ commutes
with filtered colimits and $K$, as a left adjoint, preserves all colimits.  
The map $\eta_Z$ is then a colimit of maps
$$\eta_{\tilde{Z}(A)}\colon \tilde{Z}(A) \lra JK(\tilde{Z}(A))$$
and since $\tilde{Z}(A)$ is a finite disjoint
union of corepresented functors for all $A\in {\bo P}_{I}$ thus  by assumption 
the maps $\eta_{\tilde{Z}(A)}$ are local equivalences.
This  gives that the map 
$$\hocolim[{\bo P}_{I}] \eta_{\tilde{Z}(A)}\colon \hocolim[{\bo P}_{I}] \tilde{Z}(A)\lra
 \hocolim[{\bo P}_{I}] JK(\tilde{Z}(A))$$
is a local  equivalence. It remains to notice that since ${\bo P}_{I}$ is a filtered category we have 
$\colim[{\bo P}_{I}]\tilde{Z} \simeq \hocolim[{\bo P}_{I}]\tilde{Z}$ and
$\colim[{\bo P}_{I}] JK(\tilde{Z}) \simeq \hocolim[{\bo P}_{I}] JK(\tilde{Z})$.
\qed  
\end{proof}

\begin{lemma}
\label{corep}
Let $\C$ be a multi-sorted semi-theory. 
 If $\eta_{\C_c}$ is a local equivalence for all $c \in \C$ then for any collection
$\{c_i\}_{i=1}^m$ of objects in $\C$,  $\eta_{Y}$
is a local equivalence for $Y =\coprod_{i=1}^{m} C_{c_i}.$
\end{lemma}

\begin{proof}

Let $\C$ be an $S$-sorted semi-theory. Recall that objects $c_{\underline s}\in \C$ are indexed 
by $n$-tuples ${\underline s} = (s_{1}, \dots, s_{n})$ where $s_{k}\in S$ and $n\geq 0$. In order to 
simplify notation we will write $\C_{\underline s}$ to denote the functor corepresented by 
$c_{\underline s}$. For $i=1, \dots, m$ let  ${\underline s}_{i} = (s^{i}_{1}, \dots, s^{i}_{n_{i}})$ 
be an $n_{i}$-tuple and let $Y = \coprod_{i=1}^{m} \C_{{\underline s}_{i}}$. We need to show
that $\eta_{Y}$ is a local equivalence. Let ${\underline s}$ denote the $\sum_{i=1}^{m}n_{i}$-tuple 
obtained by concatenating ${\underline s}_{i}$'s:
$${\underline s} = (s^{1}_{1}, \dots, s^{1}_{n_{1}},  \dots, s^{m}_{1}, \dots, s^{m}_{n_{m}}).$$
Notice that the projection maps in the cones of $\C$ define functors  
$$p^{{\underline s}\ast} \colon \coprod_{i=1}^{m}\coprod_{k=1}^{n_{i}} \C_{s^{i}_{k}} \to C_{\underline s}$$ 
and for $i=1, \dots, m$
$$p^{{\underline s}_{i}\ast} \colon \coprod_{k=1}^{n_{i}} \C_{s^{i}_{k}} \to C_{{\underline s}_{i}}$$ 
is given by equation (\ref{LOC MAPS}). Consider the commutative diagram
\begin{equation*}
\begin{tikzpicture}[baseline=(current bounding box.center)]
\matrix (m) 
[matrix of math nodes, row sep=5em, column sep=3em, text height=1.5ex, text depth=0.25ex]
{
\coprod_{i=1}^m \C_{{\underline s}_{i}} 
& J_{\C}K_{\C}(\coprod_{i=1}^m \C_{{\underline s}_{i}}) \\
\coprod_{i=1}^{m} \coprod_{k=1}^{n_{i}} \C_{s^{i}_{k}} 
& J_{\C}K_{\C}(\coprod_{i=1}^{m} \coprod_{k=1}^{n_{i}} \C_{s^{i}_{k}}) \\ 
\C_{\underline s }
& J_{\C}K_{\C}(\C_{\underline s }) \\
};
\path[->, thick, font=\scriptsize]
(m-1-1)
edge node[anchor=south] {$\eta$} (m-1-2)
(m-2-1)
edge node[anchor=south] {$\eta$} (m-2-2)
edge node[anchor=east] {$\coprod_{i=1}^{m} p^{{\underline s}_{i}\ast}$} (m-1-1)
edge node[anchor=east] {$p^{{\underline s}\ast}$} (m-3-1)
(m-2-2)
edge node[anchor=west] {$J_{\C}K_{\C}(\coprod_{i=1}^{m} p^{{\underline s}_{i}\ast})$} (m-1-2)
edge node[anchor=west] {$J_{\C}K_{\C}(p^{{\underline s}\ast})$} (m-3-2)
(m-3-1) 
edge node[anchor=south] {$\eta$} (m-3-2)
; 
\end{tikzpicture}
\end{equation*}
Our goal is to show that the top horizontal map is a local equivalence. Notice that the 
bottom horizontal map is a local equivalence by assumption. 
The vertical arrows on the left come from localizing maps in $\lsc$, thus they 
are local equivalences as well. As a consequence it will be enough to show that 
the vertical maps on the right are local equivalences. We will show that actually more 
is true, namely that the maps  $K_{\C}(\coprod_{i=1}^{m} p^{{\underline s}_{i}\ast})$
and $K_{\C}(p^{{\underline s}\ast})$ are isomorphism in $\bo{Alg}^{\C}$, and so that 
$J_{\C}K_{\C}(\coprod_{i=1}^{m} p^{{\underline s}_{i}\ast})$
and $J_{\C}K_{\C}(p^{{\underline s}\ast})$ are isomorphisms in $\spaces^{\C}$. To see this 
it will be enough to check that for any strict $\C$-algebra $A$ the maps induced
by $K_{\C}(\coprod_{i=1}^{m} p^{{\underline s}_{i}\ast})$
and $K_{\C}(p^{{\underline s}\ast})$ on the simplicial mapping complexes  $\map( -, A)$ are isomorphisms. 
Notice that we have 
$$\map_{\C}(K_{\C}(\C_{\underline s}), A) \cong 
\map_{\C}(\C_{\underline s}, J_{\C}A) \cong A(c_{\underline s})$$
where the first isomorphism is given by  the adjunction and the second comes 
from the Yoneda Lemma. Similarly we obtain:
$$\map_{\C}(K_{\C}(\coprod_{i=1}^m \C_{{\underline s}_{i}}), A) \cong 
\map_{\C}(\coprod_{i=1}^m \C_{{\underline s}_{i}}, J_{\C}(A)) 
\cong \prod_{i=1}^{m} A(c_{{\underline s}_{i}})$$
and
$$\map_{\C}(K_{\C}(\coprod_{i=1}^{m} \coprod_{k=1}^{n_{i}} \C_{s^{i}_{k}}), A) \cong 
\map_{\C}(\coprod_{i=1}^{m} \coprod_{k=1}^{n_{i}} \C_{s^{i}_{k}}, J_{\C}(A)) 
\cong \prod_{i=1}^{m} \prod_{k=1}^{n_{i}} A(c_{{\underline s}_{i}}).$$
Since $A$ is a strict algebra we also have isomorphisms 
$$A(c_{\underline s}) \cong \prod_{i=1}^{m} \prod_{k=1}^{n_{i}} A(c_{{ s}_{k}^i})
\cong \prod_{i=1}^{m} A(c_{{\underline s}_{i}}).$$
These isomorphisms are given by projections in $A$, so  it follows that the isomorphisms 
of the mapping complexes are induced by 
$K_{\C}(\coprod_{i=1}^{m} p^{{\underline s}_{i}\ast})$
and $K_{\C}(p^{{\underline s}\ast})$.
\qed
\end{proof}

\begin{pol}
Assume that $\eta_{\C_c}$ is a local equivalence for all $c\in \C$. It will suffice to show 
that all assumptions of Lemma  \ref{badi} are satisfied.
By Lemma \ref{corep}, $\eta_Y$ is a local equivalence for all $Y= \coprod_{i=1}^m C_{c_i},$ 
so assumption $iv)$ is satisfied.  
Assumption $iii)$ is satisfied since any weak equivalence in $\algc$ is an objectwise
weak equivalence and assumption $ii)$ is satisfied since filtered colimits in $\algc$ are computed
objectwise.  Lastly we see that $i)$ is satisfied since as in \cite[6.2]{badziochI} we get that 
if $X_{\bullet}$ is in $s\spaces^{\C}_{fib}$ then $K_{\C}|X_{\bullet}|\cong|K_{\C}X_{\bullet}|$, but
this gives $J_{\C}K_{\C}|X_{\bullet}|\cong|J_{\C}K_{\C}X_{\bullet}|.$
\qed
\end{pol}

%%%%%%%%%%%%%%%%%%%%%%%%%%%%%%%%%%%%%%%%%%%%%%%%%%%%%%%%%%%%%%
%%%%%%%%%%%%%%%%%%%%%%%%%%%%%%%%%%%%%%%%%%%%%%%%%%%%%%%%%%%%%%
%
%    5. Simplicial resolution of a multi-sorted semi-theory
%
%%%%%%%%%%%%%%%%%%%%%%%%%%%%%%%%%%%%%%%%%%%%%%%%%%%%%%%%%%%%%%
%%%%%%%%%%%%%%%%%%%%%%%%%%%%%%%%%%%%%%%%%%%%%%%%%%%%%%%%%%%%%%

\section{\bf Simplicial Resolution of a Multi-Sorted Semi-Theory}
\label{COMPLETION}

Recall that the statement of  Theorem \ref{mainIII} says that  homotopy algebras 
over an $S$-sorted semi-theory $\C$ can be rigidified to strict algebras over a 
certain simplicial $S$-sorted algebraic theory $\bFC$. In this Section we will describe 
the construction of $\bFC$.

\vspace{.35cm}
{\noindent \bf Completion of a semi-theory  }

Recall that for a set $S$ we denote by  $\bo{AlgTh}_{S}$ the category of $S$-sorted algebraic
theories and by $\bo{SemiTh}_S$ the category of $S$-sorted semi-theories (where cones are not expected to be product cones). In both categories
morphisms are functors that preserve  sorts and cones. We 
have an inclusion functor:
$$R \colon \bo{AlgTh}_{S} \lra \bo{SemiTh}_{S}.$$
By \cite[Ch.4 3.6]{Barr} every 
$S$-sorted semi-theory $\C$ has a functorially  associated $S$-sorted  algebraic theory 
$\CC$ and a morphism  $\Phi_{\C} \colon \C \ra \CC$ in $\bo{SemiTh}_{S}$ with the property that any 
functor
$$\C \lra \T$$
 to an $S$-sorted algebraic theory $\T$ uniquely factors  through $\Phi_{\C}$.
Equivalently,  this says that  the functor  $R$ has a left adjoint 
$$L \colon \bo{SemiTh}_S \ra \bo{AlgTh}_S$$  
with $\Phi_{\C}$ as the unit of adjunction. 
Furthermore, Barr and Wells showed that the following  holds.

\begin{proposition}
\label{COMPLETION ISO}
\cite[Ch.4 3.6]{Barr}
For any semi-theory $\C$ the functor $\Phi_\C\colon\C\lra\CC$
induces an equivalence of the categories of strict algebras 
$$\Phi_\C^\ast\colon\alg{\CC}\mapright{20}{\simeq}\alg{\C}$$
\end{proposition}

\begin{definition}
Given a multi-sorted semi-theory $\C$ we will call the functor  $\Phi_{\C}\colon \C \to \CC$ 
the \textit{completion of} $\C$. 
\end{definition}

\vspace{.35cm}
{\noindent \bf Simplicial Resolution\ }

Let $\C$ be a small category.  Following \cite[2.5]{dwyer-kanIII}, by the \textit{simplicial
resolution of} $\C$ we will understand the simplicial category $\bo F_\ast \C$
given as follows. ${\bo F_0\C}$ is the free category whose objects
are the objects of $\C$ and whose generators are the morphisms of $\C$.  For $k>0$
we define ${\bo F_k\C}$  to be the free category generated by ${\bo F_{k-1}\C}$.
Notice that for $c,d \in \C$ we have a canonical map 
$$\varphi_{c,d} \colon \Hom_{\C}(c,d) \lra \Hom_{\FC}(c,d).$$
If  $\C$ is an $S$-sorted semi theory then we define a simplicially enriched $S$-sorted semi-theory structure on 
$\FC$ in such way that projection morphisms of cones in $\FC$ are the  images of projections in 
$\C$ under the maps $\varphi_{c, d}$.  
Notice that in this way we have a naturally defined simplicial functor  $\psi \colon {\bo F_\ast} \C \to \C,$ which
defines a morphism of simplicial $S$-sorted semi-theories (when we regard $\C$ as a simplicial category). 
Let $\psi^{\ast}\colon \spaces^{\C} \to \spaces^{\FC}$ denote the functor induced by $\psi$, and 
let $\psi_{\ast}$ be the left adjoint of $\psi^{\ast}$.  We have the following Proposition.

\begin{proposition}
\label{C QE FC}
 The adjoint pair of functors
$$\psi_{\ast} \colon \bo L \bo{Spaces}^{\FC} \leftrightarrows \lsc \colon \psi^{\ast}$$
 is a Quillen equivalence.
\end{proposition}

\begin{proof}
 We use an argument analogous  to the proof of \cite[4.1]{badziochIII}. 
 By \cite[2.6]{dwyer-kanIII} the functor $\psi$ is a weak equivalence of categories, and so the 
 adjunction $(\psi_{\ast}, \psi^{\ast})$ is a Quillen equivalence of the model categories 
 $\spaces^{\C}_{fib}$ and $\spaces^{\FC}_{fib}$.   
Also, the morphisms with respect to which we localize $\spaces^{\FC}_{fib}$ to obtain 
$\bo{L}\spaces^{\FC}$ are precisely the
images under the functor $\psi_{\ast}$ of the localizing morphisms in  $\spaces^{\C}_{fib}$
Therefore we can apply \cite[3.3.20]{hirschhorn} which says that in such situations,  localizations 
preserve  Quillen equivalences. 
\qed
\end{proof}

For every $k\geq 0$ consider the completion 
$$ \Phi_k\colon {\bo F_k{\C}}\lra \overline{\bo F_k\C}.$$
The functors $\Phi_k$ taken together define a functor 
of simplicial categories
$$\Phi\colon\FC\lra\bFC$$
where $\bFC$ is the simplicial $S$-sorted algebraic theory which has 
$\overline{F_k\C}$ in its $k$-th simplicial dimension. 
Using Proposition \ref{COMPLETION ISO} we get the following Lemma.

\begin{lemma}
\label{ALG ISO}
The simplicial functor $\Phi\colon\FC\lra\bFC$ 
induces an equivalence of categories of strict algebras in $\spaces$
$$\Phi^\ast\colon\alg{\bFC}\overset{\simeq}{\lra}\alg{\FC}.$$
\end{lemma}

Let $\C$ be a multi-sorted semi-theory and let $\FC$ be the simplicial resolution of 
$\C$. Recall that by Proposition \ref{KC JC}  we have a Quillen pair of functors
$$\xymatrix{
K_{\FC}\colon {\bo L\sFC}\ar@<0.5ex>[r]
&\alg{\FC}\colon J_{\FC}\ar@<0.5ex>[l]
}.$$
In view of Proposition \ref{C QE FC} and Lemma \ref{ALG ISO} in order
to prove Theorem \ref{mainIII} it is enough to show that the following holds.
\begin{proposition}
\label{QUILLEN PAIR}
The Quillen pair $(K_{\FC}, J_{\FC})$
is a Quillen equivalence. 
\end{proposition}
By  Lemma \ref{compare} this fact in turn reduces to the following.
\begin{lemma}
\label{FC UNIT}
Let $\C$ be an $S$-sorted semi-theory. For an $n$-tuple $\underline s$ 
let $\FC_{\underline s} \in {\bo L \sFC}$ denote the functor
corepresented by the object $c_{\underline s} \in \FC$.
The unit of the adjunction of the pair $(K_{\FC}, J_{\FC})$
$$\eta_{\FC_{\underline s}}\colon \FC_{\underline s}\lra J_{\FC}K_{\FC}\FC_{\underline s}$$
is a local equivalence.   
\end{lemma}

\noindent The proof of Lemma \ref{FC UNIT} will be given in Section \ref{PF LEMMA} after we develop
a better understanding of the algebraic completion as well as a way of detecting 
local equivalences. Meanwhile our last goal in this Section will be to obtain an 
explicit description of the algebraic completion for free multi-sorted semi-theories.  
Our approach will parallel that of \cite[\S 3]{badziochIII}.

Let $\C$ be a free $S$-sorted semi-theory; that is $\C$ is a free category such that all projections 
in the structure cones of $\C$ are among the free generators of $\C$.  
We will construct in a combinatorial manner a category $\C'$
 and later show that $\C'$ is the  algebraic completion of $\C$ (i.e. $\C'$ is equivalent to $\CC$).   The construction of $\C'$ proceeds as follows.  
Objects $c_{\underline s} \in \C'$ are the same as the objects in $\C$ i.e. they 
are indexed by all $n$-tuples of objects of $S$ for $n\geq 0$.  As before for $s_{1}\in S$
we will denote by $c_{s_{1}}$ the object indexed by the $1$-tuple $(s_{1})$. 
 
In order to describe morphisms in $\C'$ assume first that 
$\su = (\s1,...,\s{n})$ is an arbitrary $n$-tuple and let  $s' \in S.$   
A morphism in $\Hom_{\C}(c_{\su},c_{s'})$ is  a directed tree
$T$

{\objectmargin={-1pt}
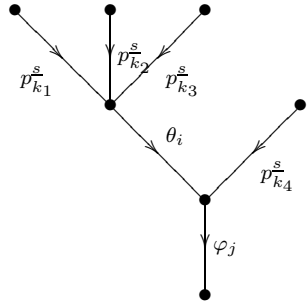
$$\xygraph{
!~:{@{-}|@{>}}
\bullet(:[dr]\bullet="a1"_{p_{k_1}^{\su}} 
:[dr]\bullet="a2"^{\theta_{i}} 
:[d]\bullet="a3"^{\varphi_{j}}) & 
\bullet(:"a1"^{p_{k_2}^{\su}}) 
& \bullet(:"a1"^{p_{k_3}^{\su}})\\ 
& &
& \bullet(:"a2"^{p_{k_4}^{\su}})
}$$}

satisfying the following conditions

\begin{itemize}
\item[1)] the lowest vertex of $T$ has only one incoming edge;
\item[2)] all edges of $T$ are labeled with $\theta_{i},$
where $\theta$ is a free generator of $\C$ whose codomain is an $n^{\theta}$-tuple in $S$, 
and where $1\leq i \leq n^{\theta}$. If $\theta = p_{k}^{\su},$ we
will write $p_{k}^{\su}$ instead of $(p_{k}^{\su})_k$;
\item[3)] if a vertex of $T$ has $m$ incoming edges with labels
$\theta^{1}_{i_{1}},...,\theta^{m}_{i_{m}}$, and for $k=1, \dots, m$ the codomain of $\theta^{k}$ in $\C$
is labeled by the $n^{\theta^{k}}$-tuple $(s_{1}^{\theta^{k}}, \dots, s_{n^{\theta_{k}}}^{\theta^{k}})$
then the outgoing edge is labeled
$\psi_{j}$ where the domain of  $\psi$ in $\C$ is labeled by the $m$-tuple
$$(s_{i_{1}}^{\theta^{1}}, s_{i_{2}}^{\theta^{2}}, \dots, s_{i_{m}}^{\theta^{m}});$$
 
\item[4)] all the initial edges of $T$ (that is, the edges starting at vertices 
with no incoming edges) are labeled with projections 
$p_{k_i}^{\su}$, where $c_{\su}$ is the domain of the tree;
\item[5)] no non-initial edges of $T$ are labeled with projection morphisms;
\item[6)] the lowest edge is labeled with $\varphi_{i}$ where the codomain of  
$\varphi$ in $\C$ is given by an $m$-tuple in $S$ whose $i$-th element is $s'$. 

\end{itemize}

For the remainder of the morphisms in $\C'$ let 
$\su =(\s1,...,\s{n})$ and $\underline{s}'  =(s_{1}',..., s_{m}')$, then 
$$\Hom_{\C'} (c_{\su},c_{\underline{s}'})=\prod_{1 \leq i \leq m} \Hom_{\C'} (c_{\su}, c_{s_{i}'}).$$

Composition of morphisms in $\C'$ will be defined the same as in \cite[\S 3]{badziochIII}: 
if $(T_{1},...,T_{m})\in \Hom_{\C'}(c_{\su}, c_{\underline s'})$
and $W \in \Hom_{\C'}(c_{\underline s'},c_{s''})$ then 
$W \circ (T_{1},...,T_{m})\in \Hom_{\C'}(c_{\su}, c_{s''})$ is
the tree obtained by grafting the tree $T_{i}$ in place of each initial 
edge of $W$ labeled $p_{i}^{\underline s'}$.  In general if 
$(T_{1},...,T_{m})\in \Hom_{\C'}(c_{\su}, c_{\underline s'})$
and $(W_{1},...,W_{r})\in \Hom_{\C'}(c_{\underline s'}, c_{\underline s''})$
then 
$$(W_{1},...,W_{r}) \circ (T_{1},...,T_{m}) 
=(W_{1} \circ (T_{\sp1},...,T_{m}),...,W_{r}
\circ (T_{1},...,T_{m}))$$ 

For an $n$-tuple $\underline s = (s_{1}, \dots, s_{n})$ in $S$ let  
${\bo p}_i^{\su}\colon c_{\underline s} \to c_{s_{i}}$  denote the morphism in $\C'$ 
represented by the tree below.
{\objectmargin={-1pt}
$$\xygraph{
!~:{@{-}|@{>}}
\bullet:[d]\bullet^{p_i^{\su}}
}$$}
\noindent We give $\C'$ an $S$-sorted semi-theory structure by choosing the morphisms
${\bo p}_{i}^{\su}$ to be projections in $\C'$. It can be checked 
that $\C'$ is, in fact,  a multi-sorted theory.

Next we define the functor 
$$\Theta_\C\colon\C\lra\C'$$ which is the identity on objects, and 
such that $\Theta_\C(p_{i}^{\su}) = {\bo p}_i^{\su}$. 
If $\varphi\colon c_{\su} \lra c_{\underline s'}$
is a generator of $\C$ which is  not a projection,  $\underline s$ is an $n$-tuple, and
$\underline s'$ is an $m$-tuple 
then  $\Theta_\C(\varphi) = (T_{1},\dots,T_{m})$ where $T_{j}$
is the following tree.

{\objectmargin={-1pt}
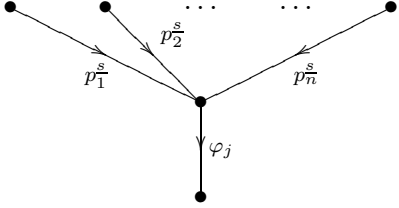
$$\xygraph{
!~:{@{-}|@{>}}
\bullet(:[drr]\bullet="u1"_{p_{1}^{\su}} 
:[d]\bullet="w1"^{\varphi_{j}}) & 
\bullet(:"u1"^{p_2^{\su}}) &
\dots & \dots &
\bullet(:"u1"^{p_n^{\su}})
}$$}

\begin{proposition}

The functor $\Theta_\C$ is the completion of the semi-theory 
$\C$ to an algebraic theory. 
\end{proposition}

\begin{proof}
By \cite[Ch.4 \S 3]{Barr} the algebraic completion of an $S$-sorted semi-theory $\C$ 
is the closure of the category $\C$ under taking finite products and this is what we constructed $\C'$ to be.
\qed
\end{proof}

%%%%%%%%%%%%%%%%%%%%%%%%%%%%%%%%%%%%%%%%%%%%%%%
%%%%%%%%%%%%%%%%%%%%%%%%%%%%%%%%%%%%%%%%%%%%%%%
%
%         6. The Initial Multi-Sorted Semi-Theory
%
%%%%%%%%%%%%%%%%%%%%%%%%%%%%%%%%%%%%%%%%%%%%%%%
%%%%%%%%%%%%%%%%%%%%%%%%%%%%%%%%%%%%%%%%%%%%%%%

\section{ \bf The Initial $S$-Sorted Semi-Theory}
\label{INITIAL}

Let $\Prj$ be the $S$-sorted semi-theory whose only non-identity 
morphisms are projections $p^{\underline s}_{k}\colon c_{\underline s}\to c_{s_{k}}$ for all 
$n$-tuples ${\underline s} = (s_{1}, \dots, s_{n})$,  $n\geq 0$ and $1\leq k \leq n.$  
Clearly $\Prj$ is an initial object in the category of $S$-sorted semi-theories,  
i.e. for any $S$-sorted semi-theory $\C$ there is a unique morphism  $\Prj\ra\C$
in $\bo{SemiTh}_{S}$.

A nice feature of $\Prj$ is that local equivalences in
${\bo L}\spaces^\Prj$ are easy to detect.

\begin{proposition}
\label{LOC WE IN PRJ}

A map $f\colon X\ra Y$ in ${\bo L}\spaces^\Prj$ is a local equivalence
if the restrictions $f_{c_{s}}\colon X(c_{s})\ra Y(c_{s})$
are weak equivalences of spaces for all $s\in S$.

\end{proposition}

\begin{proof}
Let $f\colon X \to Y$ be a map such that  $f_{c_{s}}\colon X(c_{s})\ra Y(c_{s})$
is a weak equivalence for all $s\in S$. 
Notice for any homotopy algebra $Z\in \bo {Spaces}^{\Prj}$  we can find a strict 
$\Prj$-algebra $Z'$ such that $Z'$ is fibrant in $\sprf$ and 
there is an objectwise weak equivalence $Z \ra Z'$. It follows that in order to 
show that $f$ is a local equivalence we only need to check that the induced map
of simplicial function complexes 
$$f^{\ast} \colon \map_{\Prj}(X,Z') \lra \map_{\Prj}(Y,Z')$$
is a weak equivalence for all strict $\Prj$-algebras $Z'$ that are fibrant 
in $\spaces^{\Prj}_{fib}$. This fact follows however from the observation that 
since $Z'$ is a strict algebra we have 
$$\map_{\Prj}(X,Z') \cong \prod_{s \in S} \map(X(c_{s}),Z'(c_{s}))$$
and that  under this isomorphism $f_{\ast}$ is given by the product
$\prod_{s\in S} f^{\ast}_{s}$ where
$$f^{\ast}_{s} \colon \map(Y(c_{s}), Z(c_{s})) \to  \map(X(c_{s}), Z(c_{s}))$$
is the map induced by $f_{s}$. 
\qed
\end{proof}

\begin{proposition}
\label{QP FOR PRJ C}
Let $J_\C\colon\Prj\ra\C$ denote the inclusion of $\Prj$
into a free semi-theory $\C$. Then the adjoint pair of functors
$$\xymatrix{
{J_\C}_\ast\colon \spaces^{\Prj}_{cof}\ar@<0.5ex>[r]&
\ \spaces^{\C}_{cof}\colon {J_\C}^\ast \ar@<0.5ex>[l] 
}$$
is a Quillen pair.

\end{proposition}

\begin{proof}
It can be checked that the functor ${J_\C}_\ast\colon\spaces^{\Prj}\ra\spaces^{\C}$ is left adjoint to $J_{\C^\ast}$ if we define the functor ${J_\C}_\ast$
as follows.  For objects $Y\in\spaces^{\Prj}$ let
$$ {J_\C}_\ast Y (c) = Y(c) \sqcup \coprod_{(\varphi \colon c_i \ra c) 
\in G_{c}} Y(c_i)$$
where
$c\in\C$, $G_{c}$ is the set of all morphisms 
$\varphi\colon c_i\ra c$, $c_i \in \C$ such that 
$\varphi = \zeta_k\circ\zeta_{k-1}\circ\dots\circ\zeta_1$ 
with $\zeta_1, \dots, \zeta_k $- generators of $\C$ and 
$\zeta_1$ is not a projection.
On generating morphisms (which are not projections) $\varphi \colon c\ra d$ in $\C$ we let $J_{\C_\ast}X(\varphi)$ be given by sending the component $X(c_i)$ of $J_{\C_{\ast}}X(c)$ corresponding to $\psi \colon c_i \ra c$ via the identity to the component $X(c_i)$ of $J_{\C_{\ast}}X(d)$ corresponding to $\varphi \circ \psi \colon c_i \ra d$ and sending $X(c)$ to the component $X(c)$ of $J_{\C_{\ast}}X(d)$ corresponding to $\varphi.$  For projections we do the same with the components coming from maps in $G_c$, but there is no component $X(c)$ of $J_{\C_{\ast}}X(d)$, so we instead send $X(c)$ to $X(d)$ via projection. 
Furthermore, from this description it is clear that ${J_\C}_\ast$ preserves 
objectwise cofibrations and weak equivalences. Therefore 
$({J_\C}_\ast, {J_\C}^\ast)$ is a Quillen pair.
\qed
\end{proof}

\begin{corollary}
\label{QP FOR PRJ FC}
The map of semi-theories $J\colon\Prj\ra\FC$ induces 
Quillen pair of functors
$$\xymatrix{
J_\ast\colon \spaces^{\Prj}_{cof}\ar@<0.5ex>[r]&
\ \spaces^{\FC}_{cof}\colon J^\ast \ar@<0.5ex>[l] 
}.$$

\end{corollary}

\begin{proof}
This follows by an argument paralleling the proof of  \cite[5.2]{badziochIII}.
The key idea is that the functors $(J_{k})_{\ast}$ (coming from Proposition \ref{QP FOR PRJ C} with  
$\C := F_{k}\C$) can be assembled into a functor $J_{\bullet}$ for which we have
 $J_{\ast}=|-|\circ J_{\bullet}$ where $|-|$ is the diagonal functor $|-|\colon s\spaces \ra \spaces$. 
\qed
 \end{proof}

%%%%%%%%%%%%%%%%%%%%%%%%%%%%%%%%%%%%%%%%%%%%%%%
%%%%%%%%%%%%%%%%%%%%%%%%%%%%%%%%%%%%%%%%%%%%%%%
%
%          7. Proof oF Lemma 4.17
%
%%%%%%%%%%%%%%%%%%%%%%%%%%%%%%%%%%%%%%%%%%%%%%%
%%%%%%%%%%%%%%%%%%%%%%%%%%%%%%%%%%%%%%%%%%%%%%%

\section{\bf Proof of Lemma \ref{FC UNIT}.}
\label{PF LEMMA}

Since the combinatorial description of the algebraic completion of a free multi-sorted semi-theory
uses a similar setup as the algebraic completion of free one-sorted 
 semi-theory described in \cite{badziochIII}, the remainder of the proofs
will be also similar.  For this reason  we will outline the proof of Lemma \ref{FC UNIT}, 
but refer  to \cite[\S 7]{badziochIII} for details.

Let $\D$ be a free $S$-sorted semi-theory and $\Phi_{\D}\colon\D\lra\bar{\D}$
be the completion of $\D$ to an algebraic theory. As usual we will denote objects 
of $\D$ and  $\bar{\D}$ by $d_{\underline s}$ where $\underline s$ is an $n$-tuple 
in $S$ for $n\geq 0$, and by $\D_{\underline s}$ (resp. $\bar{\D}_{\underline s}$) we will 
denote the 
$\D$-diagram (resp. the $\bar \D$-diagram) corepresented by $d_{\underline s}$. Using
the functor $\Phi_{\D}$ we can think of $\bar{\D}_{\underline s}$ as  a $\D$-diagram. 
Since $\Phi_{\D}$ is an embedding of categories $\D_{\underline s}$ is a subdiagram 
of $\bar{\D}_{\underline s}$. 

Define a filtration of the diagram $\bar{\D}_{\underline s}$ by $\D$--diagrams
$$\bardan^0\subseteq \bardan^1\subseteq\dots\subseteq\bardan$$
as follows. Set $\bardan^0 := \D_{\underline s}$. For $k\geq 0$
we define  $\bardan^{k+1}$ as  the smallest $\D$--subdiagram of $\bardan$
such that if $T_{1},T_{2},\dots,T_{m}$ are trees that are elements of 
$\bardan^k$ then 
$(T_{1},T_{2},\dots,T_{m})$ belongs to $\bardan^{k+1}$. 
From the combinatorial construction  of $\bar{\D}$ (Section \ref{COMPLETION})
we obtain that  $\colim[k]\bardan^k = \bardan$. 

As before let  $\Prj$ denote the initial $S$-sorted semi-theory. 
The unique map $\Prj\ra\D$ induces a $\Prj$-diagram structure 
on $\D_{\underline s}$, $\bardan$ and $\bardan^k$. 
We define a filtration of 
$\bardan$ by $\Prj$--diagrams 
$$s\bardan^0\subseteq s\bardan^1\subseteq\dots\subseteq\bardan,$$
\noindent where $s\D_{\underline s}^0 = \D_{\underline s}$ and $s\bardan^{k+1}$ is the smallest 
$\Prj$--subdiagram of $\bardan$ such that if 
$T_{1},T_{2},\dots,T_{m}$ are elements of $\bardan^k$  then 
$(T_{1},T_{2},\dots,T_{m})$ belongs to
 $s\bardan^{k+1}$. 
We have inclusions of $\Prj$--diagrams
$$\bardan^k\subseteq s\bardan^{k+1}\subseteq \bardan^{k+1}$$  
and  $\colim[k] s\bardan = \bardan$. 

Using the same tree-length arguments as  in \cite{badziochIII} we can check that the filtrations 
$\{\bardan^k\}$ and $\{s\bardan^k\}$ have the following
property.
 
\begin{lemma}
\label{PULLBACK LEMMA}
For any $\D$--diagram of spaces $X\colon\D\lra\spaces$
and for $k\geq 0$ the square of simplicial mapping complexes
$$\xymatrix{
\map_{\D}(\bardan^k,X)\ar[d] & \map_{\D}(\bardan^{k+1}, X)\ar[l]\ar[d] \\
\map_{\Prj}(s\bardan^k, X)     & \map_{\Prj}(s\bardan^{k+1}, X)\ar[l] \\
}$$ 
is a pull-back diagram.
\end{lemma}

Let $\C$ be an $S$-sorted semi-theory, and 
let $\FC$ be the simplicial resolution of $\C$. 
Consider $\bo F_m\C$,  the free multi-sorted semi-theory in the $m$-th 
simplicial dimension of $\FC$, and let $\overline{\bo F_m\C}$ 
be the completion of  $\bo F_m\C$  to an $S$-sorted algebraic theory. 
Setting $\D := \bo F_m\C$ above 
we see that the $\bo F_m\C$-diagram $\overline{\bo F_m\C}_{{\underline s}}$ 
(where $\overline{\bo F_m\C}_{{\underline s}}({\underline s'})=\Hom_{\overline{\bo F_m\C}}
(c_{\underline s},c_{\underline s'})$) 
admits two filtrations
by $\bo F_m\C$-diagrams:
$$ \bo F_m\C_{{\underline s}} = \overline{\bo F_m\C}_{{\underline s}}^0
\subseteq 
\overline{\bo F_m\C}_{{\underline s}}^1
\subseteq \dots\subseteq 
\overline{\bo F_m\C}_{{\underline s}}$$
and by $\Prj$--diagrams:
$$
\bo F_m\C_{{\underline s}} = 
s\overline{\bo F_m\C}_{{\underline s}}^0
\subseteq 
s\overline{\bo F_m\C}_{{\underline s}}^1
\subseteq\dots\subseteq 
\overline{\bo F_m\C}_{{\underline s}}.$$
The first of these filtrations, gives a filtration of the diagram 
$\overline{\FC}_{{\underline s}}$ by $\FC$--diagrams 
$$\FC_{{\underline s}} = \overline{\FC}_{{\underline s}}^0\subseteq \overline{\FC}_{{\underline s}}^1
\subseteq\dots\subseteq \overline{\FC}_{{\underline s}}.$$
Similarly, the filtrations of $\overline{\bo F_m \C}_{{\underline s}}$ by $\Prj$--diagrams
$s\overline{\bo F_m\C}_{{\underline s}}^k$ for $m \geq 0$ give a filtration of 
$\overline{\FC}_{{\underline s}}$ by $\Prj$-diagrams
$$\FC_{{\underline s}} = s\overline{\FC}_{{\underline s}}^0\subseteq s\overline{\FC}_{{\underline s}}^1
\subseteq\dots\subseteq \overline{\FC}_{{\underline s}}.$$

We have the following Lemma.

\begin{lemma}
\label{DIAG}
For $X\in\spaces^{\FC}$ consider the following diagrams of 
simplicial function complexes
$$\xymatrix{
\map_{\FC}(\overline{\FC}_{{\underline s}}^k,X)\ar[d] & 
\map_{\FC}(\overline{\FC}_{{\underline s}}^{k+1}, X)\ar[l]_f\ar[d] \\
\map_{\Prj}(s\overline{\FC}_{{\underline s}}^k, X)  & 
\map_{\Prj}(s\overline{\FC}_{{\underline s}}^{k+1}, X)\ar[l]_g \\
}$$ 
This is a pullback diagram for all $X$, and $k\geq 0$ and $\alpha \in \tau.$
\end{lemma}
 
\begin{proof}
This follows directly from Lemma \ref{PULLBACK LEMMA} and 	
\cite[6.1]{badziochIII}.
\qed 
\end{proof}

Next we wish to show that the map $g$ from 
Lemma \ref{DIAG} satisfies the following. 

\begin{lemma}
\label{LOWER MAP}
Let $X$ be a homotopy algebra fibrant in $\spaces_{cof}^{\FC}$.
For every $k\geq0$
the map 
$$g\colon \map_{\Prj}(s\overline{FC}_{{\underline s}}^{k+1}, X)\lra
\map_{\Prj}(\overline{FC}_{{\underline s}}^k, X)$$
induced by an inclusion 
$\iota_k\colon\overline{FC}_{{\underline s}}^k\hookrightarrow s\overline{FC}_{{\underline s}}^{k+1}$
is an acyclic fibration of simplicial sets.
\end{lemma}
 
\begin{proof}
 Since $\iota_k$ is a cofibration in $\spaces_{cof}^{\FC}$
we get that $g$ is a fibration.  It remains to show that $g$ 
is also a weak equivalence 
of simplicial sets. 

By Corollary \ref{QP FOR PRJ FC}, if $X$ is a homotopy $\FC$-algebra 
fibrant in $\spaces_{cof}^{\FC}$ then it is also a 
homotopy $\Prj$-algebra which is fibrant in $\spaces_{cof}^\Prj$. 
Therefore we need only show that the map $\iota_k$
is a local equivalence in $\spaces^\Prj$, but this is a result of 
Theorem \ref{LOC WE IN PRJ} and the fact that
$\iota_k$ restricts to an isomorphism of simplicial sets
$$\overline{\FC}_{{\underline s}}^k(s)\mapright{25}{\cong} 
s\overline{\FC}_{{\underline s}}^{k+1}(s)$$
for all $s \in S.$

\qed
\end{proof}

Next, consider the upper map $f$ in the diagram in Lemma \ref{DIAG}. 
From Lemma \ref{LOWER MAP} we have that $g$ is an acyclic fibration and
from Lemma \ref{DIAG} $f$ is the base change of $g$ along 
$$\map_{\FC}(\overline{\FC}_{{\underline s}}^k,X) \lra \map_{\Prj}(\overline{\FC}_{{\underline s}}^k,X)$$
so by \cite[3.14]{dwyerspalinski} we obtain.

\begin{corollary}
\label{AC FIB}
Let $X$ be a homotopy algebra, fibrant in $\spaces_{cof}^{\FC}$.
For all $k\geq 0$ the map 
$$f\colon\map_{\FC}(\overline{\FC}_{{\underline s}}^{k+1}, X)\lra
\map_{\FC}(\overline{\FC}_{{\underline s}}^k,X)$$
is an acyclic fibration of simplicial sets.
\end{corollary}

We can now  give the proof of Lemma \ref{FC UNIT}.

\begin{polo}
The map  $\eta_{\FC_{{\underline s}}}\colon \FC_{{\underline s}}\ra J_\C K_\C \FC_{{\underline s}}$ 
is given by the inclusion of $\FC$-diagrams 
$$\FC_{{\underline s}} = \overline{\FC}_{{\underline s}}^0\hookrightarrow \overline{\FC}_{{\underline s}}=
J_{\C}K_{\C}\FC_{{\underline s}}$$ 
Moreover,  
$\overline{\FC}_{{\underline s}} = \colim[k]\overline{\FC}_{{\underline s}}^k.$ 
We have a commutative diagram 
$$\xymatrix{
\FC_{{\underline s}} \ar[dr]_{\eta_{\FC_{{\underline s}}}}\ar[rr] & & 
\hocolim[k] \overline{\FC}_{{\underline s}}^k \ar[dl] \\
 & \colim[k]\overline{\FC}_{{\underline s}}^k & \\
}$$

We have that the top map can be given by the composition
$$\FC_{{\underline s}} \ra \hocolim[k] \FC_{{\underline s}} \ra \hocolim[k] \overline{\FC}_{{\underline s}}^k,$$
the first of which is a local equivalence since $\hocolim[k]\FC_{{\underline s}} \simeq \FC_{{\underline s}}\otimes [0,\infty)$ and the second map is a local equivalence since the induced map $$\holim[k] \map(\overline{\FC}_{{\underline s}}^k,Z)\ra \holim[k]\map(\FC_{{\underline s}},Z)$$ is a weak equivalence by \cite[18.5.1]{hirschhorn} and the fact that the homotopy limits are taken over diagrams which are separated by objectwise acyclic fibrations.  We also see that the bottom right map of the above diagram is a local equivalence since both the homotopy colimit and colimit can be computed objectwise.  Furthermore, for $c\in\C$ we have that $$\FC_{{\underline s}}^0(c)\hookrightarrow \FC_{{\underline s}}^1(c) \hookrightarrow \cdots$$ is a cofibrant diagram, so by \cite[18.9.4]{hirschhorn} the map from the homotopy colimit to the colimit is a local equivalence.
Therefore the map $\eta_{FC_{{\underline s}}}$ is also 
a local equivalence.
\qed
\end{polo}

Recall that   Lemma \ref{FC UNIT} was the last element we needed to complete the proof of 
Theorem \ref{mainIII}. Therefore  Theorem \ref{mainIII} is now established.

%%%%%%%%%%%%%%%%%%%%%%%%%%%%%%%%%%%%%%%%%%%%%%%
%%%%%%%%%%%%%%%%%%%%%%%%%%%%%%%%%%%%%%%%%%%%%%%
%
%          8. Proof of Theorem 3.4
%
%%%%%%%%%%%%%%%%%%%%%%%%%%%%%%%%%%%%%%%%%%%%%%%
%%%%%%%%%%%%%%%%%%%%%%%%%%%%%%%%%%%%%%%%%%%%%%%

\section{ \bf Proof of Theorem \ref{mainIV}}
\label{PF THM}

We will proceed with
a proof of Theorem \ref{mainIV}.  To start we will take note of
the following Lemma which can be seen to be analogous to Theorem \ref{MAINII} for multi-sorted algebraic theories (see also \cite[8.6]{rezk}).

\begin{lemma}
\label{EQ FOR ALG}
Let $\T$ and $\T'$ be multi-sorted algebraic theories and $G\colon \T\ra\T'$ be a  
functor of finite product sketches. 
Then $G$ induces an adjoint pair of functors
between the categories of strict algebras 
$$\xymatrix{
G_\ast\colon \alg{\T}\ar@<0.5ex>[r] &
\alg{\T'} \colon G^\ast \ar@<0.5ex>[l] \\
}$$ 
which is a Quillen pair. Moreover, each of the following is true
\begin{itemize}
\item[i)] If $G$ is a weak r-equivalence, then the induced functor $G^{\ast}$
 gives an equivalence of the homotopy categories of strict algebras.
\item[ii)] If we assume that $G$ is surjective on the sets of objects, then
 $G^\ast$ gives an 
equivalence of the homotopy theories of strict algebras
iff the functor $G$ is a weak equivalence of categories.
\end{itemize}
\end{lemma}

\begin{proof}
The adjoint pair $(G_{\ast},G^{\ast})$ exists by \cite[ch.4 3.5]{Barr}, we also 
see that $G^{\ast}$ preserves fibrations and weak equivalences since both are computed
objectwise and $G$ preserves products so we in fact have a Quillen pair. 

Now, to prove part $ii)$, if we have that $G^{\ast}$ gives an equivalence of homotopy categories,
then in particular its left adjoint, $G_{\ast}$ gives the inverse.
For $T_{\su} \in \T$ we have the corepresented diagram $\Hom_{\T}(T_{\su},-)$ which is
cofibrant in $\alg{\T}.$  By using these equivalences we get that the unit 
$$\eta \colon  \Hom_{\T}(T_{\su},-) \lra G^{\ast}G_{\ast}\Hom_{\T}(T_{\su},-)$$
is an objectwise weak equivalence.  By the property of adjunction 
and Yoneda's Lemma we also get:
$$\map_{\T'}(G_{\ast}\Hom_{\T}(T_{\su},-),X)$$
$$\cong \map_{\T}(\Hom_{\T}(T_{\su},-),G^{\ast}X)$$
$$\cong G^{\ast}X(T_{\su}) \cong X(G(T_{\su} ))$$
$$\cong \map_{\T'}(\Hom_{\T'}(G(T_{\su} ),-),X)$$
for all $X \in {\alg{{\T '}}}$, but this gives
$$G_{\ast}\Hom_{\T}(T_{\su} ,-) \cong \Hom_{\T'}(G(T_{\su}),-)$$
and
$$G^{\ast}G_{\ast}\Hom_{\T}(T_{\su} ,-) \cong \Hom_{\T'}(G(T_{\su}),G(-)).$$
In particular the unit for a corepresented diagram
given by
$$\eta \colon \Hom_{\T}(T_{\su} ,-) \lra \Hom_{\T'}(G(T_{\su}) ,G(-))$$
is an objectwise weak equivalence, but this gives that $G$ is a weak equivalence if we further assume that $G$ is surjective on the sets of objects.

Now we prove $i).$  We need only show that if $G$ is a weak r-equivalence, then $G^{\ast}$ gives
an equivalence of homotopy algebras.  
The proof of this direction follows the proof of \cite[3.4]{Schwede} and the fact that 
if $G$ is a weak r-equivalence of categories, 
the unit of adjunction for any corepresented diagram is an objectwise weak equivalence.  Recall from
the proof of Proposition \ref{ALGTHMODSTR PROP} 
that we have a pair of adjoint functors
$$F\colon \spaces^{S} \lra \alg{\T} \colon U$$
It can be seen that this is in fact a Quillen pair since in the model category of $\spaces^S,$ weak
equivalences, fibrations, and cofibrations are all computed objectwise.  By \cite[9.6]{May} for any 
strict $\T$-algebra $X,$ we can define a 
simplicial object $\mathcal{B}(X)$ in $s\alg{\T}$ by letting $\mathcal{B}(X)_n=(FU)^{n+1}(X).$  Furthermore,
by \cite[9.8]{May} we see that the geometric realization of $\mathcal{B}(X),$ denoted $|\mathcal{B}(X)|$ has 
the property that there exists a map $\varphi \colon |\mathcal{B}(X)| \lra X$ which is a weak equivalence of
strict $\T$-algebras.  By following an analogous argument to the one given in the proof of \cite[3.6]{badziochI}, we 
also see that $|{\mathcal{B}}(X)|$ is a cofibrant object of $\alg{\T}$, thus $|{\mathcal{B}}(X)|$ is a cofibrant replacement
of $X.$  

Now, we see that any cofibrant diagram $X$ can be replaced by $|\mathcal{B}(X)|,$ for which it can be seen that 
the unit of adjunction $\eta_{|\mathcal{B}(X)|}$ is an objectwise weak equivalence if the unit of adjunction for 
any corepresented diagram is an objectwise weak equivalence.  This follows from the argument given in the proof
of \cite[3.4]{Schwede}, which in turn gives us that the unit of adjunction $\eta_{X}$ is an objectwise weak equivalence.  In addition, since $G$ is a weak r-equivalence, every object in $\T'$ is the retract of an object in the image of $G$.  Using a retract of maps argument this gives that for any function $\varphi \in \alg{\T'}$, $\varphi$ is an objectwise weak equivalence if and only if $G^{\ast}\varphi$ is an objectwise weak equivalence.  This, in combination with the fact that the unit of adjunction for any cofibrant diagram is an objectwise weak equivalence, gives us that $(G_{\ast},G^{\ast})$ is a Quillen equivalence.
\qed
\end{proof}

\noindent With that we can give the proof of Theorem \ref{mainIV}:

\begin{poto}
Let $G\colon \C_1\ra\C_2$ be a functor of multi-sorted semi-theories. 
Consider the (non-commutative) diagram 
$$\xymatrix{
\alg{\bFC_1} \ar@<0.5ex>[r] \ar@<0.5ex>[d] &
\alg{\bFC_2}\ar@<0.5ex>[l]\ar@<0.5ex>[d]\\
{\bo L\sFCO} \ar@<0.5ex>[r]\ar@<0.5ex>[d]\ar@<0.5ex>[u]& 
{\bo L\spaces^{\FC_2}}\ar@<0.5ex>[l]\ar@<0.5ex>[d]\ar@<0.5ex>[u]\\ 
\lsco \ar@<0.5ex>[r]\ar@<0.5ex>[u] & 
{\bo L\spaces^{\C_2}} \ar@<0.5ex>[l]\ar@<0.5ex>[u]\\ 
}$$
in which every pair of arrows represents a Quillen pairs of functors. 
The horizontal pairs are induced by the functor $G$ while 
the vertical ones come from the adjunctions of (\ref{C QE FC}), 
(\ref{KC JC}), and (\ref{ALG ISO}). 
Propositions \ref{C QE FC} and \ref{QUILLEN PAIR} imply
that the vertical pairs are Quillen equivalences. 
By this we have that $G$ induces an equivalence of the homotopy categories of
$\lsco$ and  ${\bo L_{S}\spaces^{\C_2}}$ iff it induces 
an equivalence of the homotopy categories of strict algebras
$\alg{\bFC_1}$ and $\alg{\bFC_2}$. Thus Lemma \ref{EQ FOR ALG}
completes the proof.
\qed
\end{poto}

%%%%%%%%%%%%%%%%%%%%%%%%%%%%%%%%%%%%%%%%%%%%%%%%%%%%%%%
%%%%%%%%%%%%%%%%%%%%%%%%%%%%%%%%%%%%%%%%%%%%%%%%%%%%%%%
%
%        10. The Associated Multi-Sorted Semi-Theory
%
%%%%%%%%%%%%%%%%%%%%%%%%%%%%%%%%%%%%%%%%%%%%%%%%%%%%%%%
%%%%%%%%%%%%%%%%%%%%%%%%%%%%%%%%%%%%%%%%%%%%%%%%%%%%%%%

\section{\bf The Associated Multi-Sorted Semi-Theory}
\label{assmsst}

Our next goal will be to prove Theorems \ref{MAINI COR} and \ref{MAINII}. Our strategy will 
be to use a reductive process: we will show that an arbitrary finite product sketch can be 
replaced by a multi-sorted semi-theory which has the same homotopy category of homotopy 
algebras. As a result, Theorems \ref{MAINI COR} and \ref{MAINII} will follow directly 
from their already established analogs for multi-sorted semi-theories, i.e. Theorems \ref{mainIII} 
and \ref{mainIV}.

The reduction of finite product sketches to multi-sorted semi-theories will be performed in 
two stages. First, we will show that for any finite product sketch we can construct a
multi-sorted finite product sketch in a way that preserves the the homotopy theory of 
homotopy algebras.

\begin{definition}
Let $S$ be a set. 
 An \emph{$S$-sorted finite product sketch} is a finite product sketch $(\C, \kappa)$ with a 
 distinguished set of objects $\{c_{s}\}_{s\in S}$ indexed by $S$
 with the following properties
\begin{itemize}
 \item for all $\alpha \in \kap$ and $i> 0$ we have $\alpha_i \in \{c_{s}\}_{s\in S}$;
 \item for all $\alpha \in \kap,$ $\alpha_0 \notin \{c_{s}\}_{s\in S}$ (unless 
 $|\alpha|=1$ and $p_1^\alpha = \rm id$);
 \item if $\alpha, \beta\in \kappa$, $|\alpha|=n= |\beta|$,  
 and $\{\alpha_i\}_{i=1}^n = \{\beta_i\}_{i=1}^n$ then $\alpha = \beta$;
 \item if $\alpha, \beta\in \kappa$ and $\alpha_0 = \beta_0$ then $\alpha = \beta$. 
\end{itemize}
A \emph{multi-sorted finite product sketch} is a finite product sketch that is $S$-sorted for some set $S$. 
\end{definition}

\begin{lemma}
\label{BB}
 For any  finite product sketch $\B$ there exists  a multi-sorted finite product sketch $\BBM$ such that 
the homotopy categories of homotopy algebras over $\B$ and $\BBM$ are equivalent.  
Moreover, this construction is functorial.
\end{lemma}

The proof of Lemma \ref{BB} is a consequence of the following fact. 
Given two  $n$-fold cones $\alpha$ and $\beta$ we will say that these cones are isomorphic 
if there exists a natural transformation between them given by the set of maps $\{f_i\}_{i=0}^{n}$ 
where $f_i\colon \alpha_{i} \ra \beta_i$  is an isomorphism for all $i$. We have the following Lemma.

\begin{lemma}
\label{midB}
 Suppose $(\B_1,\kappa_1)$ and $(\B_2,\kappa_2)$ are two finite product sketches. 
 Assume also that we have a functor  
$$F\colon \B_1 \lra \B_2$$ 
such that $F$ is an equivalence of categories. Assume also that the following conditions hold:

\begin{itemize}
 \item for each $\alpha \in \kappa_1,$ there is  $   \alpha ' \in \kappa_2$ so that $F(\alpha) \cong \alpha ',$
 \item for each $ \alpha ' \in \kappa_2 ,  $ there is $ \alpha \in \kappa_1 $ so that $\alpha ' \cong F(\alpha).$
 \end{itemize} 
 Then the Quillen pair of functors 
$$F_{\ast} \colon \lsddo \leftrightarrows \lsddt \colon F^{\ast}$$
is a Quillen equivalence. 
\end{lemma}

\begin{proof}
First, since $F$ is an equivalence of categories it induces a Quillen equivalence 
$$F_{\ast} \colon \spaces^{\B_1}_{fib} \leftrightarrows \spaces^{\B_2}_{fib} \colon F^{\ast}.$$
Since $F$ preserves cones (up to an isomorphism) we get  
that maps which we localize $\spaces^{\B_2}$ by are sent by $F^{\ast}$ to maps
which we localize $\spaces^{\B_1}_{fib}$ by (up to isomorphism).  By \cite[3.3.20]{hirschhorn} it 
follows that  
$$F_{\ast} \colon \lsddo \leftrightarrows \lsddt \colon F^{\ast}$$ is a Quillen equivalence.

\qed
\end{proof}

We can now proceed to the proof of Lemma \ref{BB}.
\begin{polt}
 Let $(\B, \kappa)$ be a finite product sketch.
We define the category  $\kapp$ which has as its objects,
the set of cones $\{\alpha | \alpha \in \kap \}$ and for each pair of objects 
$\alpha , \beta \in \kapp$ there is a unique isomorphism 
$$\psi_{\alpha, \beta} \colon \alpha \lra \beta.$$

Next, let 
$J$ be the category with two objects $0$ and $1$ and the  non-identity morphisms given by two inverse 
isomorphisms: $$\varphi \colon 0 \leftrightarrows 1 \colon \varphi^{-1}.$$
Take $\BBM = \B \times \kapp \times J$. For every $n$-fold cone $\alpha\in \kappa$ we have 
the associated $n$-fold cone in $\BBM$ given as follows:
$$\xymatrix{
& (\az, \alpha, 0) \ar[d]_{{\rm id}\times {\rm id} \times \varphi} & & \\
& (\az, \alpha, 1) \ar[ddl]_{p^{\alpha}_1 \times {\rm id}\times {\rm id}}
\ar[dd]^{p^{\alpha}_2 \times {\rm id}\times {\rm id}}\ar[ddrr]^{p^{\alpha}_n \times {\rm id}\times {\rm id}} & & \\
& & & & \\
(\ao, \alpha, 1)  & 
(\alpha_2, \alpha, 1)  & \dots & 
(\alpha_n, \alpha, 1) \\
}$$
Consider the sketch $(\BBM, \kappa_{\mu})$ where $\kappa_{\mu}$ is the set of all cones of the above form along with the identity cones for each object of $\BBM$ which does
not show up in such a cone. 
Notice that $(\BBM, \kappa_{\mu})$ is a multi-sorted finite product sketch with the distinguished set of 
objects $\{(\alpha_{i}, \alpha, 1)\}_{\alpha\in \kappa,\  0< i \leq |\alpha|}$. 
We define the functor: $$F \colon \B \lra \BBM,$$
by  $F(b)=(b,\alpha,0)$ and 
$F(\theta \colon b_1 \ra b_2)= \theta \times {\rm id} \times {\rm id},$ where
$\alpha$ is some fixed cone from $\kappa$ (we can assume that 
$\kappa$ is non-empty since $\B$ would satisfy Lemma \ref{BB} trivially otherwise).  
It can be checked that  $F$ satisfies the conditions of  Lemma \ref{midB}, and so 
it gives a Quillen equivalence 
$$F_{\ast} \colon \lsd \leftrightarrows \lsddm \colon F^{\ast}.$$
Thus $F^{\ast}$ induces an equivalence between  homotopy categories of homotopy algebras 
over $\B$ and $\BBM$.
\qed 
\end{polt}

Next, we will show that any multi-sorted finite product sketch can be 
replaced by a multi-sorted semi-theory in a way that does not change the 
homotopy theory of homotopy algebras. 

\begin{lemma}
\label{BBB}
For any multi-sorted finite product sketch $(\B, \kappa)$ there exists a multi-sorted semi-theory 
$(\BBS, \kappa_{\sigma})$ so that 
the homotopy category of homotopy algebras over $\B$ and $\BBS$ are equivalent.  Moreover, this 
construction is functorial in $\B$.
\end{lemma}

\begin{proof}
Notice that a multi-sorted finite product sketch $(\B, \kappa)$ can be equivalently described as follows. 
There exists a set $S$ such that objects  $b_{\underline s}\in \B$ can be indexed by \emph{some}
of the $n$-tuples of $S$ ($n\geq 0$) and for any $s\in S$ we have $b_{s}\in S$ (as before we identify 
here elements of $S$ with 1-tuples defined by these elements). 
Any cone $\alpha\in \kappa$ satisfies the property that if 
$\alpha_{0} = b_{\underline s}$ where $\underline s = (s_{1}, \dots, s_{n})$ then $\alpha_{k}= b_{s_{k}}$ 
for $k=1, \dots, n$. Moreover, for any $b_{\underline s}\in \B$  there exists a unique cone $\alpha\in \kappa$
such that $\alpha_{0} = b_{\underline s}$. In other words the difference between $\B$ and 
an $S$-sorted semi-theory is that for some $n$-tuples $\underline s$ in $S$ there may be 
no object of $\B$ indexed by $\underline s$, and thus the cone corresponding to $\underline s$
will be also missing. To fix it we enlarge that category $\B$ as follows. Let $\underline S$ denote the set of all 
$n$-tuples in $S$:
$$\underline S = \{\underline s = (s_{1}, \dots, s_{n}) \ | \ s_{i} \in S, n\geq 0\}.$$
Also, let $\underline S_{\B}$ denote the set of all $n$-tuples that index elements of $\B$:
$$\underline S_{\B} = \{ \underline s \in \underline S \ | \  b_{\underline s} \in \B \}.$$  
We start by letting $\BBS$ be the smallest category whose objects $b_{\underline s}$ are indexed by all elements 
$\underline s\in \underline S$, with morphisms determined so that $\B$ is a full subcategory of $\BBS$ and for each 
$\underline s = (s_{1}, \dots, s_{n}) \not\in \underline S_{\B}$ the category $\BBS$ has morphisms 
$p^{\underline s}_{k}\colon b_{\underline s} \to b_{s_{k}}$, which compose freely with 
morphisms in $\B$. We give $\BBS$ a finite product sketch structure by defining the set of cones $\kappa_{\sigma}$
that consists of all cones in $\kappa$ and for each 
$\underline s= (s_{1}, \dots, s_{n}) \not \in \underline S_{\B},$ a cone
$\alpha^{\underline s}$ with  $\alpha^{\underline s}_{0} = b_{\underline s}$, for $k=1, \dots, n$
$\alpha^{\underline s}_{k} = b_{s_{k}}$ and with projections given by the morphisms   
$p^{\underline s}_{k}$. Clearly $(\BBS, \kappa_{\sigma})$ is an $S$-sorted semi-theory. 

It remains to show that  the homotopy category of homotopy algebras over $(\B, \kappa)$
is equivalent to the homotopy category of homotopy algebras over $(\BBS, \kappa_{\sigma})$. 
Let $F\colon \B \to \BBS$ be the inclusion functor, and let $F^{\ast}\colon \spaces^{\BBS}\to \spaces^{\B}$
be the functor induced by $F$.  The functor $F$ has a right adjoint $G$ which can be described 
as follows. For $X\in \spaces^{\B}$ the functor $G(X)\colon \BBS\to \spaces$ coincides with $X$
when restricted to $\B\subseteq \B'$. For an object $b_{\underline s}\in \BBS$ such that 
$b_{\underline s}\not\in \B$ and where $\underline s = (s_{1}, \dots, s_{n})$ we set 
$$G(X)(b_{\underline s}) = \prod_{k=1}^{n}X(b_{s_{k}}).$$

Denote by $ \alg{\B}_{h}$ and $ \alg{\BBS}_{h}$ the full subcategories 
of $\spaces^{\B}$ and $\spaces^{\BBS}$ respectively whose objects are homotopy $\B$-(resp. $\BBS$-)algebras.  Notice that both  $F^{\ast}$ and $G$ restrict to functors 
$$F^{\ast} \colon \alg{\BBS}_{h} \leftrightarrows \alg{\B}_{h}\colon G$$
Notice that the composition $F^{\ast}G$ is naturally isomorphic to the identity functor via the counit. Also, for any 
$X\in \alg{\BBS}_{h}$ the unit of adjunction is an objectwise weak equivalence
$X \overset{\simeq}{\lra}  GF^{\ast}(X)$. Since the homotopy categories of homotopy algebras 
are obtained from $\alg{\B}_{h}$ and $\alg{\BBS}_{h}$ by inverting all objectwise weak equivalences 
we obtain that $F^{\ast}$ and $G$ give inverse equivalences on the level of the homotopy categories. 
\qed
\end{proof}

%%%%%%%%%%%%%%%%%%%%%%%%%%%%%%%%%%%%%%%%%%%%%%%%%%%%%%%%%%%%%%%%%%%%%%%%%%%%%%
%
%           11. Proof of Theorem 1.12 and Theorem 1.13
%
%
%%%%%%%%%%%%%%%%%%%%%%%%%%%%%%%%%%%%%%%%%%%%%%%%%%%%%%%%%%%%%%%%%%%%%%%%%%%%%%%%

\section{\bf  Proof of Theorem \ref{MAINI COR} and Theorem \ref{MAINII}}
\label{bigpf}
 
We can now use Lemma \ref{BB} and lemma \ref{BBB} to see that any finite product sketch $\B$ can be replaced
by a multi-sorted semi-theory $\BBB$ which has an equivalent homotopy category of homotopy algebras and use this to prove Theorem \ref{MAINI COR}.
  
\begin{pottc}
Let $\B$ be a finite product sketch.  We can compose the equivalences from Lemma \ref{BB} and Lemma \ref{BBB} to get an equivalence of homotopy categories
$$\Homcat\lsd \leftrightarrows
\Homcat\lsddm \leftrightarrows \Homcat \bo{Alg}_h^{\BBM}
\leftrightarrows \Homcat\bo{Alg}_h^{\BBB}.$$
The middle equivalence comes from the fact that we constructed $\lsddm$ to serve as a model category for $\bo{Alg}_h^{\BBM}.$
This gives us that the homotopy category of homotopy $\B$-algebras
is equivalent to the homotopy category of homotopy $\BBB$-algebras.
Since $\BBB$ is a multi-sorted semi theory we have an equivalence from Theorem \ref{mainIII},
$$ \Homcat{\bo{L}} \spaces^{\BBB} \leftrightarrows \Homcat\bo{Alg}^{\overline{F_{\ast}\BBB}}.$$
By composing with the equivalences above we get the following equivalence of homotopy categories.
$$\Homcat\lsd \leftrightarrows \Homcat\bo{Alg}^{\overline{F_{\ast}\BBB}}$$
\noindent which gives us that the homotopy category of homotopy $\B$-algebras
is equivalent to the homotopy category of strict $\overline{F_{\ast}\BBB}$-algebras.
\qed
\end{pottc}

With this we can proceed with a proof of Theorem \ref{MAINII}.
\begin{potth}
For two finite product sketches $\BO$ and $\BT$, let $\BOM$, $\BTM,$ $\BOMS,$ and $\BTMS$ be the associated
multi-sorted finite product sketches and multi-sorted semi-theories from Lemma \ref{BB} and Lemma \ref{BBB}
respectively.
Notice that for a functor of finite product sketches which is an injection on objects:
$$F\colon \BO \lra \BT$$ we can induce a cone preserving functor:
$$F^{\mu} \colon \BOM \lra \BTM$$
$$(B,\alpha , i) \lra (F(B),F(\alpha),i)$$
where $F(\alpha)$ makes sense because $F$ preserves cones.  This induces a functor:
$$F^{\mu \sigma}\colon \BOMS \lra \BTMS$$ 
which is defined naturally on objects of the form $(B,\alpha , i)$ and which sends any added cones
in $\BOMS$ to the obvious cone in $\BTMS.$  We see that this is a functor of multi-sorted theories.  By construction we see that $F^{\mu \sigma}$ induces an equivalence of the homotopy categories of homotopy algebras if and only if $F$ does the same.  The remainder of the proof is therefore a consequence of Theorem \ref{mainIV} since Theorem \ref{mainIV} is essentially Theorem \ref{MAINII} for multi-sorted semi theories.  Furthermore, it can be seen that the assumptions on objects and cones in Theorem \ref{mainIV} lift to the associated assumptions on objects and cones for Theorem \ref{MAINII}.
\qed
\end{potth}

In summary, given a finite product sketch $(\B,\kappa)$ we can consider the homotopy structures that arise from homotopy algebras over $(\B,\kappa).$  In many situations these homotopy structures can be interesting to study (for example the infinite loop space structures that arise from considering homotopy algebras over $\Gamma^{op}$ \ref{GAMMA EXAMPLE}) and it is worth determining if there is some weakly equivalent strict algebraic structure that also arises.  This paper has shown that such an equivalence exists if we consider strict algebras over the associated simplicial multi-sorted algebraic theory $\FB$.  This gives our rigidification of homotopy algebras over $(\B,\kappa).$  Furthermore, we have given a way of determining if two finite product sketches determine equivalent homotopy categories of homotopy algebras by considering the associated simplicial multi-sorted algebraic theories.

\vskip .5cm  

\noindent{\bf Acknowledgment. \ } This paper is a version of my Ph.D. thesis completed at the State University of New York
at Buffalo.  I would like to thank my advisor
Bernard Badzioch for his incredible patience and for his valuable suggestions.
I would like to thank my loving wife for her encouragement and support.  I would also like to thank the referee for his or her helpful suggestions and comments.


\begin{thebibliography}{10}

\bibitem{badziochI}
Bernard Badzioch.
\newblock Algebraic theories in homotopy theory.
\newblock {\em Ann. of Math. (2)}, 155(3):895--913, 2002.

\bibitem{badziochIII}
Bernard Badzioch.
\newblock From {$\Gamma$}-spaces to algebraic theories.
\newblock {\em Trans. Amer. Math. Soc.}, 357(5):1779--1799 (electronic), 2005.

\bibitem{Barr}
Michael Barr and Charles Wells.
\newblock Toposes, triples and theories.
\newblock {\em Repr. Theory Appl. Categ.}, (12):x+288, 2005.
\newblock Corrected reprint of the 1985 original [MR0771116].

\bibitem{Bergner}
Julia~E. Bergner.
\newblock Rigidification of algebras over multi-sorted theories.
\newblock {\em Algebr. Geom. Topol.}, 6:1925--1955, 2006.

\bibitem{bousfield}
A.~K. Bousfield.
\newblock The simplicial homotopy theory of iterated loop spaces.
\newblock Manuscript, 1992.

\bibitem{Dugger}
Daniel Dugger.
\newblock Replacing model categories with simplicial ones.
\newblock {\em Trans. Amer. Math. Soc.}, 353(12):5003--5027 (electronic), 2001.

\bibitem{dwyer-kanIII}
W.~G. Dwyer and D.~M. Kan.
\newblock Simplicial localizations of categories.
\newblock {\em J. Pure Appl. Algebra}, 17(3):267--284, 1980.

\bibitem{dwyer-kanI}
W.~G. Dwyer and D.~M. Kan.
\newblock Equivalences between homotopy theories of diagrams.
\newblock In {\em Algebraic topology and algebraic {$K$}-theory ({P}rinceton,
  {N}.{J}., 1983)}, volume 113 of {\em Ann. of Math. Stud.}, pages 180--205.
  Princeton Univ. Press, Princeton, NJ, 1987.

\bibitem{dwyerspalinski}
W.~G. Dwyer and J.~Spali{\'n}ski.
\newblock Homotopy theories and model categories.
\newblock In {\em Handbook of algebraic topology}, pages 73--126.
  North-Holland, Amsterdam, 1995.

\bibitem{Farjoun}
Emmanuel~Dror Farjoun.
\newblock {\em Cellular spaces, null spaces and homotopy localization}, volume
  1622 of {\em Lecture Notes in Mathematics}.
\newblock Springer-Verlag, Berlin, 1996.

\bibitem{Fiore}
Thomas~M. Fiore.
\newblock Pseudo limits, biadjoints, and pseudo algebras: categorical
  foundations of conformal field theory.
\newblock {\em Mem. Amer. Math. Soc.}, 182(860):x+171, 2006.

\bibitem{goerss}
Paul~G. Goerss and John~F. Jardine.
\newblock {\em Simplicial homotopy theory}, volume 174 of {\em Progress in
  Mathematics}.
\newblock Birkh\"auser Verlag, Basel, 1999.

\bibitem{hirschhorn}
Philip~S. Hirschhorn.
\newblock {\em Model categories and their localizations}, volume~99 of {\em
  Mathematical Surveys and Monographs}.
\newblock American Mathematical Society, Providence, RI, 2003.

\bibitem{lawvere}
F.~William Lawvere.
\newblock Functorial semantics of algebraic theories.
\newblock {\em Proc. Nat. Acad. Sci. U.S.A.}, 50:869--872, 1963.

\bibitem{May}
J.~P. May.
\newblock {\em The geometry of iterated loop spaces}.
\newblock Springer-Verlag, Berlin-New York, 1972.
\newblock Lectures Notes in Mathematics, Vol. 271.

\bibitem{prezma}
Matan Prezma.
\newblock Homotopy normal maps.
\newblock {\em Algebr. Geom. Topol.}, 12(2):1211--1238, 2012.

\bibitem{rezk}
Charles Rezk.
\newblock Every homotopy theory of simplicial algebras admits a proper model.
\newblock {\em Topology Appl.}, 119(1):65--94, 2002.

\bibitem{Schwede}
Stefan Schwede.
\newblock Stable homotopy of algebraic theories.
\newblock {\em Topology}, 40(1):1--41, 2001.

\bibitem{segal}
Graeme Segal.
\newblock Categories and cohomology theories.
\newblock {\em Topology}, 13:293--312, 1974.

\end{thebibliography}
\end{document}